\documentclass[a4paper,11pt,leqno]{amsart}
\usepackage{amsmath,amsthm,amssymb,amsfonts,graphicx,mathrsfs,bm,amscd,hyperref,latexsym,color}
\usepackage[all]{xy}

\numberwithin{equation}{section}
\theoremstyle{plain}
\newtheorem{thm}{Theorem}[section]
\newtheorem{lem}[thm]{Lemma}
\newtheorem{prop}[thm]{Proposition}
\newtheorem{cor}[thm]{Corollary}
\newtheorem{rem}[thm]{Remark}
\newtheorem*{convention}{Convention}

\theoremstyle{remark}

\def\R{\mathbb{R}}
\def\Z{\mathbb{Z}}
\def\N{\mathbb{N}}
\def\B{\mathbb{B}}
\def\S{\mathbb{S}}
\def\W{\mathcal{W}}

\def\d{\delta}

\def\a{\alpha}
\def\T{\mathbb{T}}

\newcommand{\dist}{\text{dist}}

\newcommand{\id}{\mathrm{id}}
\DeclareMathOperator{\diam}{diam}

\newcommand{\cX}{\mathcal{X}}
\newcommand{\cS}{\mathcal{S}}

\newcommand{\cW}{\mathcal{W}}
\newcommand{\cM}{\mathcal{M}}
\newcommand{\cT}{\mathcal{T}}
\newcommand{\Mod}{\operatorname{Mod}}

\newcommand{\haus}{\operatorname{\mathcal{H}}}
\newcommand{\interior}{\mathrm{int}}
\newcommand{\fI}{\mathfrak{I}}
\newcommand{\Bd}{\mathrm{Bd}}
\newcommand{\Wh}{\mathrm{Wh}}
\newcommand{\HW}{\mathrm{HW}}
\newcommand{\sS}{\mathscr{S}}

\begin{document}

\title[Quasiconformal non-parametrizability]{Quasiconformal non-parametrizability of almost smooth spheres}
\author{Pekka Pankka}
\address{Department of Mathematics and Statistics, P.\,O.\, Box 35 (MaD), FI-40014
University of Jyv\"{a}skyl\"a, Finland \and 
Department of Mathematics and Statistics, P.\,O.\, Box 68, FI-00014
University of Helsinki, Finland}
\email{pekka.pankka@helsinki.fi}
\author{Vyron Vellis}
\address{Department of Mathematics and Statistics, P.\,O.\, Box 35 (MaD), FI-40014
University of Jyv\"{a}skyl\"a, Finland \and Department of Mathematics, University of Connecticut, 341 Mansfield Road U1009
Storrs, CT 06269, USA}
\email{vyron.vellis@uconn.edu}
\date{\today}
\thanks{P.P. and V.V. were supported by the Academy of Finland (projects 283082 and 257428).}
\keywords{quasiconformal gauge, quasisymmetric non-parametrization, almost smooth Riemannian metric, decomposition space}
\subjclass[2010]{30L10 (30C65)}

\begin{abstract}
We show that, for each $n\ge 3$, there exists a smooth Riemannian metric $g$ on a punctured sphere $\S^n\setminus \{x_0\}$ for which the associated length metric extends to a length metric $d$ of $\S^n$ with the following properties: the metric sphere $(\S^n,d)$ is Ahlfors $n$-regular and linearly locally contractible but there is no quasiconformal homeomorphism between $(\S^n,d)$ and the standard Euclidean sphere $\S^n$.
\end{abstract}

\maketitle

\section{Introduction}
\label{sec:intro}

The \emph{(quasi)conformal gauge} of the Euclidean $n$-sphere $\S^n$ is the maximal collection of all metrics $d$ on $\S^n$ for which there exists a quasiconformal homeomorphism $(\S^n,d) \to \S^n$. Here, and in what follows, $\S^n$ refers to both the subset $\{ (x_1,\ldots, x_{n+1})\in \R^{n+1}\colon x_1^2 + \cdots + x_{n+1}^2 =1\}$ of $\R^{n+1}$ but also the metric space $(\S^n,d_0)$, where $d_0$ is the Euclidean metric induced from $\R^{n+1}$ by inclusion. 

The problem of characterizing this gauge is a relaxation of the Beltrami problem in the analytic theory of quasiconformal mappings. Indeed, whereas the Beltrami problem asks whether a given measurable Riemannian metric $g$ on $\S^n$ admits conformal map $(\S^n,g) \to (\S^n,g_0)$ into the standard Riemannian metric $g_0$, the gauge characterization problem merely asks for a quasiconformal map between metrics. We refer to Heinonen  \cite[Section 15]{HeinonenJ:Lecams} for a detailed discussion on the terminology and background of the (quasi)conformal gauge. 

Characterization of the quasiconformal gauge has turned out to be a formidable problem, and the question remains open also for the quasisymmetric gauge in higher dimensions; the \emph{quasisymmetric gauge} of the Euclidean sphere $\S^n$ consists of all metric spheres $(\S^n,d)$ admitting a quasisymmetric homeomorphism $(\S^n,d)\to \S^n$; see Section \ref{sec:pre} for terminology.

In dimensions $n=1$ and $n=2$ the quasisymmetric gauge is fully understood. For $\S^1$ the metric characterization for the quasisymmetric gauge is due to Tukia and V\"ais\"al\"a \cite{TukiaP:Quaems} and for $\S^2$ this gauge is characterized by Bonk and Kleiner \cite{BonkM:Quaptd}; see also Wildrick \cite{WildrickK:Quaptd, WildrickK:Quasos}. In particular, all Ahlfors $2$-regular and linearly locally contractible (LLC) metric $2$-spheres $(\S^2,d)$ are quasisymmetrically equivalent to $\S^2$. We note in passing that Ahlfors $n$-regular and LLC metric spheres $(\S^n,d)$ are $n$-Loewner spaces and quasiconformal maps $(\S^n,d)\to \S^n$ are quasisymmetric; see Heinonen-Koskela \cite{HeinonenJ:Quamms}. We refer to Rajala \cite{RajalaK:Uni2d} for recent results on quasiconformal parametrization of metric $2$-spheres.

In higher dimensions these metric conditions are not sufficient for quasisymmetric parametrization. By results of Semmes \cite{SemmesS:Goomsw} (dimension $n=3$) and Heinonen--Wu \cite{HW} (dimensions $n>3$), \emph{there exists for each $n\ge 3$ an Ahlfors $n$-regular, LLC, and geodesic $n$-sphere, which is not quasisymmetrically equivalent to the standard sphere $\S^n$.}

The metric sphere $(\S^3,d_S)$ Semmes considered in \cite{SemmesS:Goomsw} is the decomposition space $\S^3/\Bd$ associated to the Bing double and the metric $d_S$ is obtained by an embedding $\S^3/\Bd \to \S^4$; see Section \ref{sec:Dec} for definitions. Heinonen and Wu consider in \cite{HW} the decomposition space $\R^3/\Wh$, associated to the Whitehead continuum, and construct an Ahlfors $3$-regular and linearly locally contractible metric $d_{\HW}$ on $\R^3/\Wh$. For $n>3$, the stabilization $\R^3/\Wh\times \R^{n-3}$ is homeomorphic to $\R^n$ and a product metric, also denoted by $d_{\HW}$, in the stabilized space $\R^3/\Wh\times \R^{n-3}$ is Ahlfors $n$-regular and LLC. A metric sphere $(S_{\HW}^n,d_{\HW})$, which is not in the quasisymmetric gauge of $\S^n$, is now obtained by one-point compactification of $(\R^3/\Wh\times \R^{n-3},d_{\HW})$.

Neither the sphere $\S^3/\Bd$ nor the spheres $S_{\HW}^n$ have \emph{a priori} smooth structures; choices of homeomorphisms $\S^3/\Bd \to \S^3$ and $S_{\HW}^n \to \S^n$ introduce such on these spaces. Note that there exists a homeomorphism $\S^3/\Bd\to \S^3$ and a Cantor set $C\subset \S^3/\Bd$ for which the domain $(\S^3/\Bd)\setminus C$ is diffeomorphic to a domain in the standard sphere $\S^3$. Under this parametrization of $\S^3/\Bd$, we may take the metric $d_S$ to be the completion of a Riemannian distance in $(\S^3/\Bd)\setminus C$; see Section \ref{sec:BB} for details. Similarly, in the Heinonen--Wu example $(S_{\HW}^n, d_{\HW})$ there exists a codimension $3$ sphere $S$ for which $S_{\HW}^n\setminus S$ is diffeomorphic to an open subset of $\S^n$ and for which $d_{HW}$ is a completion of the distance $d_g$ associated to a Riemannian metric $g$ in $S_{\HW}^n\setminus S$.

We say that a length metric $d$ on $\S^n$ is \emph{almost smooth} if there exists a compact set $E\subset \S^n$ (called \emph{singular set}) and a smooth Riemannian metric $g$ in $\S^n\setminus E$ for which $d$ is the completion of the distance $d_g$ associated to $g$. Recall that a metric $d$ on $\S^n$ is a \emph{length metric} if $d(x,y) = \inf_\gamma \ell_d(\gamma)$ for all points $x,y\in  \S^n$, where $\gamma$ is a path connecting $x$ and $y$ and $\ell_d(\gamma)$ is the length of $\gamma$ in metric $d$. 

We show that, for each $n\ge 3$, there exists an almost smooth metric $d$ on $\S^n$ having a singular set consisting of only one point but for which there is no quasisymmetric homeomorphism $(\S^n,d)\to \S^n$.

\begin{thm}
\label{thm:point}
For each $n \ge 3$ there exists an almost smooth Ahlfors $n$-regular and linearly locally contractible length metric $d$ in $\S^n$ with a singular set consisting of a single point $x_0\in \S^n$ for which there is no quasiconformal homeomorphism $(\S^n,d)\to \S^n$.
\end{thm}

It should be noted that, although the singular set of the metric $d$ consists only of one point $x_0\in \S^n$, the quasiconformal non-parametrization of the sphere $(\S^n,d)$ stems from the degeneration of the underlying Riemannian metric. Indeed, the metric $d$ we construct for Theorem \ref{thm:point} has the property that there is no quasiconformal homeomorphism $(\S^n\setminus \{x_0\},d) \to \S^n\setminus \{x_0\}$. We refer to Balogh and Koskela \cite{BaloghZ:Quaqua} for removability results for quasiconformal mappings to the positive direction in the metric setting.

The construction in Theorem \ref{thm:point} is based on Blankinship's necklace \cite{Blankinship}, a higher-dimensional analogue of Antoine's necklace which yields a wild Cantor set in $\S^n$. In the proof of Theorem \ref{thm:point} we use a modification of this construction for two rings, which we call the \emph{Bing--Blankinship construction} since it gives a generalization of Bing's double to higher dimensions. To obtain an almost smooth sphere with one singular point, we consider a sequence of partial Bing--Blankinship constructions of arbitrary length. 

The non-existence of a quasiconformal homeomorphism $(\S^n,d)\to \S^n$ is based on uniform modulus estimates for certain families of $(n-2)$-tori. These modulus estimates stem from uniform area estimates which replace Semmes's length estimates in \cite{SemmesS:Goomsw}. These area estimates are obtained by homological intersection counting in the spirit of Freedman and Skora \cite[Lemma 2.5]{FreedmanM:Strags}; see Proposition \ref{prop:vie} and Corollary \ref{cor:vie}. 

These modulus estimates, when applied to the decomposition space associated to the Bing--Blankinship construction, yield a sharp higher dimensional metric analog of Semmes's non-parametrizability result \cite{SemmesS:Goomsw} for metrics on $\S^3$. We discuss this result (Theorem \ref{thm:main}) and its relation to the result of Heinonen and Wu in Section \ref{sec:proof_main}.

This article is organized as follows. In Section \ref{sec:BB}, we discuss the Bing--Blankinship decomposition space $\S^n/\mathrm{BB}$ and show that $\S^n/\mathrm{BB}$ is homeomorphic to $\S^n$. We also construct the almost smooth metric $d$ in Theorem \ref{thm:main}. In Section \ref{sec:viemaps} we prove Freedman--Skora intersection results for the decomposition associated to $\S^n/\mathrm{BB}$. In Section \ref{sec:modulus} we discuss modulus estimates in $(\S^n,d)$ and in the Euclidean sphere $\S^n$ for families of $(n-2)$-tori associated to decomposition yielding $\S^n/\mathrm{BB}$. In Section \ref{sec:proof_main} we prove Theorem \ref{thm:main} and finally Theorem \ref{thm:point} in Section \ref{sec:proof}.

\subsection*{Acknowledgments} We thank Jang-Mei Wu for discussions on the shrinkability of the Bing--Blankinskip necklace.

\section{Preliminaries}
\label{sec:pre}

We begin this section with a general discussion on the metric theory of quasiconformal mappings and Loewner spaces. As a second topic we recall notions from point set topology related to decomposition spaces. We finish this section with a discussion on Semmes metrics on decomposition spaces. 

\subsection{Loewner spaces and quasiconformal maps}

A homeomorphism $f\colon X\to Y$ between metric spaces $(X,d_X)$ and $(Y,d_Y)$ is \emph{quasiconformal} if there exists $H<\infty$ satisfying
\begin{equation}
\label{eq:H}
\limsup_{r\to 0} \frac{\sup_{d_X(x,y)\le r} d_Y(f(x),f(y))}{\inf_{d_X(x,y)\ge r} d_Y(f(x),f(y))} \le H
\end{equation}
for every $x\in X$. A homeomorphism $f\colon X\to Y$ is \emph{$\eta$-quasisymmetric}, where $\eta\colon [0,\infty) \to [0,\infty)$ is a homeomorphism, if
\[
\frac{d_Y(f(x),f(y))}{d_Y(f(x),f(z))} \le \eta\left( \frac{d_X(x,y)}{d_X(x,z)}\right) 
\]
for all triples $x,y,z\in X$ of distinct points in $X$.

The spaces we consider in this article are Ahlfors $n$-regular and linearly locally contractible. A metric measure space $(X,d,\mu)$ is \emph{Ahlfors $n$-regular} if there exists a constant $C>0$ for which 
\[
\frac{1}{C} r^n \le \mu(B_X(x,r)) \le C r^n 
\]
for all open metric balls $B_X(x,r) = \{y \in X\colon d(x,y)<r\}$ of radius $0<r \le \diam{X}$ about $x$ in $X$. We call a metric space $(X,d)$ \emph{Ahlfors $n$-regular} if $(X,d,\mathcal{H}^n)$ is $n$-regular; here and in what follows $\mathcal{H}^n$ is the Hausdorff $n$-measure with respect to the metric $d$. Further, $X$ is \emph{linearly locally contractible} if there exists $C>0$ so that each ball $B_X(x,r)$ in $X$ is contractible in $B_X(x,Cr)$ for all $r<(\diam X)/C$.

Connected and orientable $n$-manifolds that are Ahlfors $n$-regular and linearly locally contractible support $(1,n)$-Poincar\'e inequality; see \cite[Theorem B.10]{SemmesS:Fincgs}. Thus, when proper, they are $n$-Loewner spaces by a result of Heinonen and Koskela \cite[Theorem 5.7]{HeinonenJ:Quamms}. Further, a quasiconformal homeomorphism between bounded Ahlfors $n$-regular spaces ($n>1$) is quasisymmetric if the domain is a Loewner space and the target linearly locally contractible \cite[Theorem 4.9]{HeinonenJ:Quamms}.

A space $(X,d,\mu)$ is \emph{$n$-Loewner} if there exists nonincreasing positive function $\phi \colon (0,\infty) \to (0,\infty)$ such that
\begin{equation}
\label{eq:Loewner}
\Mod_n(\Gamma(E,F)) \geq \phi(t) > 0
\end{equation}
whenever $E$ and $F$ are two disjoint, non-degenerate continua in $X$ and
\[
t \geq \frac{\dist(E,F)}{\min(\diam{E}, \diam{F})}.
\]
Here $\Mod_n (\Gamma(E,F))$ is the $n$-modulus of the family $\Gamma(E,F)$ of all paths connecting $E$ and $F$. 

Recall that the \emph{$p$-modulus} $\Mod_p(\Gamma)$, for $p\ge 1$, of a path family $\Gamma$ in $(X,d,\mu)$ is 
\[
\Mod_p(\Gamma) = \inf \int_X \rho^p \, \mathrm{d}\mu,
\]
where the infimum is taken over all non-negative Borel functions $\rho \colon X \to [0,\infty]$ satisfying
\[
\int_\gamma \rho \,\mathrm{d}s \geq 1
\] 
for all locally rectifiable paths $\gamma\in \Gamma$.
 
More generally, given a family $\cS$ of $l$-manifolds (possibly with boundary) in $X$, where $l \in  \{1,\ldots, n-1\}$, the \emph{$p$-modulus} $\Mod_p(\cS)$ of $\cS$ is 
\[
\Mod_p(\cS) = \inf \int_X \rho^p \, \mathrm{d}\mu,
\]
where the infimum is taken over all non-negative Borel functions $\rho\colon X\to [0,\infty]$ satisfying
\begin{equation}\label{eq:admis}
\int_{S} \rho(x) \, \mathrm{d}\mathcal{H}^l \ge 1
\end{equation}
for all $S\in \cS$.
A function $\rho$ satisfying \eqref{eq:admis} is called an \emph{admissible function for $\cS$}. 

The proofs of Theorems \ref{thm:point} and \ref{thm:main} are based on the quasi-invariance of the $n/(n-2)$-modulus of a family of $(n-2)$-manifolds. More precisely, we consider the modulus of a family $\Sigma = \{ \{x\}\times (\S^1)^{n-2} \subset \B^2\times (\S^1)^{n-2} \colon x\in \Omega\}$ in $\B^2\times (\S^1)^{n-2}$, where $\Omega\subset \B^2$ is a neighborhood of the boundary of $\B^2$. Given a quasiconformal map $f\colon \Omega\times (\S^1)^{n-2}\to U$, where $U\subset \R^n$ is a domain, we have
\begin{equation}
\label{eq:Sigma_Mod}
\frac{1}{C_0} \Mod_{\frac{n}{n-2}}(\Sigma) \le \Mod_{\frac{n}{n-2}}(f\Sigma) \le C_0 \Mod_{\frac{n}{n-2}}(\Sigma),
\end{equation}
where $C_0=C_0(n,H)>0$ depends only on $n$ and the quasiconformality constant $H$ of $f$; see \cite[Theorem 6]{AgardS:Quamm}. The $n/(n-2)$-modulus of $\Sigma$ is conformally invariant and is therefore called the \emph{conformal modulus of $\Sigma$}. 

\begin{rem}
We note, in passing, that our definition for the $p$-modulus of a family of $l$-manifolds is slightly more restrictive than the definition given in Agard \cite{AgardS:Quamm} as we have used the Hausdorff $l$-measure in the definition of admissbile functions in place of more general $l$-dimensional measure. We also refer to Rajala \cite{RajalaK:Surfbb} for a definition of modulus for more general families of geometric sets.
\end{rem}

\subsection{Decomposition spaces, initial packages, and defining sequences}
\label{sec:Dec}

We introduce now the topological notions of a decomposition space and a defining sequence which will be used throughout the article. The defining sequences we consider are induced by (Semmes's) initial packages; see \cite[Section 2]{SemmesS:Goomsw} for a detailed discussion on initial packages. 

A \emph{decomposition} of topological space $X$ is a partition of $X$. The partitions $G$ we consider are \emph{upper semi continuous (usc)}, that is, elements of $G$ are closed subsets of $X$ and for each $g\in G$ and every neighborhood $U$ of $g$ in $X$ there exists a neighborhood $V$ of $g$ contained in $U$ so that every $g'\in G$ intersecting $V$ is contained in $U$. A \emph{defining sequence} $\cX=(X_k)_{k\ge 0}$ for the decomposition $G$ is a decreasing sequence of closed sets $X$ for which the components of $\bigcap_{k\ge 0} X_k$ are exactly the non-degenerated elements in $G$, that is, elements $g\in G$ which are not points. 

The \emph{decomposition space $X/G$ associated to $G$} is the quotient space with the quotient topology induced by the canonical map $X\to X/G$; for usc decompositions, $X/G$ is a metrizable space. We refer to Daverman \cite{DavermanR:Decm} for a detailed discussion on decomposition spaces.

A tuple $\fI=(M; \varphi_1,\ldots, \varphi_p)$ is a \emph{smooth initial package} if $M$ is a compact manifold with boundary and each $\varphi_i \colon M \to M$ is a smooth embedding with the property that images $\varphi_i(M)$ are pair-wise disjoint. Initial packages $\fI = (M;\varphi_1,\ldots, \varphi_p)$ and $\fI'=(M';\varphi'_1,\ldots, \varphi'_p)$ are \emph{equivalent} if there exists a diffeomorphism $\beta \colon M\to M'$ for which $\varphi'_i = \beta \circ \varphi_i \circ \beta^{-1}$; we say that diffeomorphism $\beta$ \emph{conjugates $\fI$ and $\fI'$}. 

Each initial package $\fI=(M;\varphi_1,\ldots, \varphi_p)$ induces a natural tree ordered by inclusion. More precisely, let $\cW_p$ be the set of all finite words in alphabet $\{1,\ldots,p\}$. For each $w=i_1\cdots i_k\in \cW_p$, let $\varphi_w \colon M\to M$ be the embedding $\varphi_w := \varphi_{i_1}\circ \cdots \circ \varphi_{i_k}$; for $w=\emptyset$, we set $\varphi_{\emptyset}=\id$. Then $\varphi_w(M) \supset \varphi_{wi}(M)$ for each $i=1,\ldots, p$ and $w\in \cW_p$. We call the ordered tree $\cT_\fI = (\varphi_w(M))_{w\in \cW_p}$ the \emph{defining tree of $\fI$}.

The \emph{defining sequence for $\fI$} is the sequence $\cX_\fI= (\cX_{\fI,k})_{k\ge 0}$ where $\cX_{\fI,k} =  \bigcup_{|w|=k} \varphi_w(M)$ and $|w|$ is the length of the word $w$. Note that, by construction, for each $w\in \cW_p$ of length $k$, the set $\varphi_w(M)$ is a component of $\cX_{\fI,k}$. The \emph{decomposition $G_\fI$ of $M$} is the decomposition associated to the defining sequence $\cX_\fI$. We denote the set $\bigcap_k \cX_{\fI,k}$, the \emph{singular set of the initial package $\fI$}, by $\cS(\fI)$.

In forthcoming sections we consider defining sequences, which are not defining sequences for initial packages, but are topologically equivalent to such defining sequences. For this purpose, we say that $\cM=(M_w)_{w\in \cW_p}$ is an \emph{ordered tree} if $M_{wi}\subset M_w$ for each $w\in \cW_p$ and $i=1,\ldots, p$. Ordered trees $\cM=(M_w)_{w\in \cW_p}$ and $\cM'=(M'_w)_{w\in \cW_p}$ are \emph{equivalent} if there exists a homeomorphism $\theta \colon M_\emptyset \to M'_\emptyset$ satisfying $\theta(M_w)=M'_w$ for each $w\in \cW_p$. Further, we say that an ordered tree $\cM$ is \emph{equivalent to the initial package $\fI$} if $\cM$ is equivalent to the tree $\cT_\fI$.

\begin{convention}
Let $(M;\varphi_1,\ldots, \varphi_p)$ be an initial package,  $E\subset M$ a set, and let $w\in \cW_p$ be a word. Then $E_w\subset M$ is the set $E_w = \varphi_w(E)$ unless otherwise specified. 
\end{convention}

\subsection{Semmes metrics}

The metric $d$ on $\S^n$ in Theorem \ref{thm:point} stems from a construction of a quasi-self-similar metric on a decomposition space of $\S^n$. These metrics are introduced in \cite{SemmesS:Goomsw} and called Semmes metrics in \cite{PW}. They are defined as follows.
 
A metric $d$ on a decomposition space $M/G_{\fI}$ associated to an initial package $\fI$ is a \emph{Semmes metric} if there exists $\lambda >0$ and $L\ge 1$ for which 
\[
\frac{\lambda^k}{L} d(x,y) \le d(\varphi_w(x),\varphi_w(y)) \le L\lambda^k d(x,y)
\]
for each $x,y\in M/G_{\fI}$ and each word $w = i_1\cdots i_k$; we call the metric space $(M/G_{\fI}, d)$ a \emph{(self-similar)} Semmes space and $\lambda$ the \emph{scaling constant of the metric $d$}. We refer to \cite[Section 7]{PW} for Semmes metrics and Semmes spaces associated to non-self-similar decomposition spaces.

Metric spaces $(M/G_{\fI},d)$ are Ahlfors $n$-regular and LLC under mild conditions on the metric $d$ and the initial package $\fI$. We record these facts as lemmas. The proofs are minor variations of the proofs of \cite[Proposition 7.8]{PW} and \cite[Proposition 7.9]{PW}, respectively.

\begin{lem}
\label{lem:Ahlfors}
Let $M$ be an $n$-manifold with boundary for $n\ge 3$, $\fI=(M;\varphi_1,\ldots, \varphi_p)$ an initial package, and $d$ a Semmes metric on $M/G_{\fI}$ with a scaling constant $\lambda \in (0,p^{-1/n})$. Then $(M/G_{\fI},d)$ is Ahlfors $n$-regular.
\end{lem}

\begin{lem}
\label{lem:LLC}
Let $M$ be an $n$-manifold with boundary for $n\ge 3$, $\fI=(M,\varphi_1,\ldots, \varphi_p)$ an initial package, and $d$ a Semmes metric on $M/G_{\fI}$. Suppose $\varphi_{w i}(M)$ contracts in $\varphi_w(M)$ for each $w\in \cW_p$ and $i\in \{1,\ldots,p\}$. Then $(M/G_{\fI},d)$ is linearly locally contractible.
\end{lem} 

\section{Bing--Blankinship spheres}
\label{sec:BB}

In this section, we introduce first the construction of the Bing--Blankinship necklace which yields the decomposition space $\S^n/\mathrm{BB}$. In the construction we present in this, we combine the idea of Blankinship in \cite{Blankinship} based on Antoine's necklace with a contruction of Bing in \cite{BingR:Homb3s}. We also adapt Bing's method to show that the space $\S^n/\mathrm{BB}$ is homeomorphic to $\S^n$.

\subsection{The decomposition space $\S^n/\mathrm{BB}$}
 
Let $n \ge  3$ and $\psi \colon (\S^1)^{n-2}\to (\S^1)^{n-2}$ be the cyclic permutation 
\[
(x_1,x_2,\ldots,x_{n-2}) \mapsto (x_{n-2},x_1,x_2,\ldots, x_{n-3}),
\]
where we understand $\psi = \id$ for $n=3$. Let also $\fI_B=(\B^2\times \S^1; \varphi_1,\varphi_2)$ be an initial package for the Bing double; see Figure \ref{fig:Bing}. 

\begin{figure}[h!]
\includegraphics[scale=0.55]{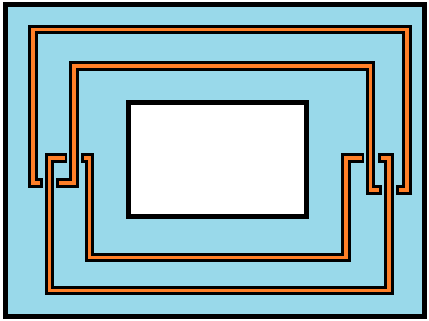}
\caption{The Bing double.}

\label{fig:Bing}
\end{figure}

A generalization of the initial package $\fI_B$ to dimension $n$ is 
\[
\fI_{B,n} = (\B^2\times (\S^1)^{n-2}; \tilde \varphi_1,\tilde \varphi_2),
\]
where 
\[
\tilde \varphi_i = (\varphi_i \times \id_{(\S^1)^{n-3}}) \circ (\id_{\B^2}\times\psi).
\]
We call the initial package $\fI_{B,n}$ the \emph{Bing--Blankinship package}; note that $\fI_{B,3} =  \fI_B$. As in \cite{Blankinship}, we call a space homeomorphic to $\T = \B^2\times (\S^1)^{n-2}$ an \emph{$n$-tube}. We call $n$-tubes 
\[
\T_w=\tilde \varphi_w(\B^2\times (\S^1)^{n-2})
\]
for $w\in \cW_2$, \emph{Blankinship rings}.

To obtain a decomposition of $\S^n$, we fix a smooth embedding $\vartheta \colon \T \to \S^n$ for which $\vartheta(\T)\subset \R^n\subset \S^n$ and there exists $x_0\in \S^1$ satisfying
\begin{itemize}
\item[(i)] $\vartheta(\B^2\times \S^1\times \{x_0\}^{n-3})\subset \R^3\times \{0\}^{n-3}$,
\item[(ii)] $\vartheta(\varphi_i(\B^2\times \S^1)\times\{x_0\}^{n-3}) \subset \R^3\times \{0\}^{n-3}$ for $i=1,2$, and
\item[(iii)] $\vartheta \circ (\id_{\B^2}\times \psi^k) = (\id_{\R^2} \times \psi')^k \circ \vartheta$ for each $k\in \N$, where $\psi' \colon \R^{n-2} \to \R^{n-2}$ is the cyclic permutation $(x_1,\ldots, x_{n-2}) \mapsto (x_{n-2},x_1,\ldots, x_{n-3})$.
 
\end{itemize}

The decomposition space $\S^n/\mathrm{BB}$ is now the decomposition space obtained by collapsing the components of $\vartheta(\cS(\fI_{B,n}))$. Thus there exists a natural embedding $\vartheta' \colon \T/G_{\fI_{B,n}} \to \S^n/\mathrm{BB}$ satisfying
\[
\xymatrix{
\T \ar[r]^{\vartheta} \ar[d] & \S^n \ar[d] \\
\T/G_{\fI_{B,n}} \ar[r]^{\vartheta'} & \S^n/\mathrm{BB}
}
\]
where vertical arrows are canonical maps $x \mapsto [x]$.

In the following two sections, we show that the space $\S^n/\mathrm{BB}$ admits a good embedding into the Euclidean sphere $\S^{n+1}$ and that $\S^n/\mathrm{BB}$ is topologically an $n$-sphere. Using this two observations, we construct in Section \ref{sec:almost_smooth_metric} an almost smooth metric on the standard $n$-sphere $\S^n$ having the Bing--Blankinship necklace as the singular set.

\subsection{Embedding of $\S^n/\mathrm{BB}$ into $\S^{n+1}$}
 
The space $\S^n/\mathrm{BB}$ admits a embedding into $\S^{n+1}$, which is modular in the terminology of \cite{PW}.

\begin{prop}
\label{prop:BB_Semmes}
Let $n\ge 3$ and $\lambda\in (0,1)$. Then there exists a map $\tilde \rho \colon \S^n \to \S^{n+1}$ with the following properties:
\begin{itemize}
\item[(i)] $\tilde\rho|\left(\S^n\setminus \cS(\fI_{B,n})\right) \colon \S^n\setminus \cS(\fI_{B,n}) \to \S^{n+1}$ is a smooth embedding for which $\tilde \rho(x)=x$ for every $x\in \S^n\setminus \vartheta(\T)$, and
\item[(ii)] there exists $L\ge 1$ so that, for each word $w\in \cW_2$, the map $\tilde \rho_w = \tilde \rho \circ \vartheta \circ \tilde\varphi_w \circ \vartheta^{-1}|\vartheta(\T\setminus (\T_1\cup \T_2)) \colon \vartheta(\T\setminus (\T_1\cup \T_2)) \to \S^{n+1}$ satisfies
\[
\frac{\lambda^{|w|}}{L} |x-y| \le |\tilde \rho_w(x)-\tilde \rho_w(y)| \le L \lambda^{|w|}|x-y|
\]
for all $x,y\in \vartheta(\T\setminus (\T_1\cup \T_2))$.
\end{itemize}
In particular, $\tilde \rho(\cS(\fI_{B,n}))$ is a Cantor set in $\S^{n+1}$ and there exists an embedding $\S^n/\mathrm{BB}\to \S^{n+1}$ for which the diagram 
\[
\xymatrix{
\S^n \ar[rr]^{\tilde\rho} \ar[dr]^{x\mapsto [x]} && \S^{n+1} \\
& \S^n/\mathrm{BB} \ar[ur] & }
\]
commutes.
\end{prop}

The proof of Proposition \ref{prop:BB_Semmes} is a minor modification of the argument of Semmes in \cite[Lemma 3.21]{SemmesS:Goomsw} and we merely sketch the proof; see also \cite[Section 6]{PW} for a similar construction. 

\begin{proof}[Sketch of a proof of Proposition \ref{prop:BB_Semmes}]
Let $T_w = \vartheta(\T_w)$ for each $w\in \cW_2$. Let also $\tilde \varphi'_i = \vartheta \circ \tilde\varphi_i \circ \vartheta^{-1}|T_\emptyset$ for $i=1,2$. Then $\fI_{B,n}'=(T_\emptyset; \tilde \varphi'_1,\tilde \varphi'_2)$ is an initial package.

The construction of the map $\tilde \rho$ is self-similar with respect to the initial package $\fI_{B,n}'$, and we describe only the first step. We may assume that $T=T_\emptyset \subset \R^n \subset \S^n$. Let $\mu = \sum_{k=0}^\infty \lambda^k$ and let $B\subset \R^n$ be a Euclidean ball containing $T$ for which $\dist(T,\R^n\setminus B) \ge 2\mu$. For each $w\in \cW_2$ we also denote by $C_w$ the cylinder $T_w \times [-\lambda^{|w|}, \lambda^{|w|}] \subset \R^{n+1}$. For simplicity, we denote $C=C_\emptyset$ and $U = B\times [-2\mu,2\mu]$.

Let $g_i \colon C \to \R^{n+1}$, for $i=1,2$, be $\lambda$-similarities for which $g_i(T)\subset \R^n\times \{\lambda\}$ and the images $g_1(T)$ and $g_2(T)$ are pair-wise disjoint. Then there exists a diffeomorphism $G_1 \colon U \to U$ for which $G_1|\partial U = \id$, $G_1|\partial T = \id$ and $G_1(\tilde \varphi'_i(x),\lambda s) = g_i(x,s)$ for $(x,s)\in C$. The existence of $G_1$ follows from Semmes's unlinking argument in dimension $3$. 

Indeed, recall that we have assumed that there exists $x_0\in \S^1$ for which $t=\vartheta(\B^2\times \S^1 \times \{x_0\}^{n-3})$ and $t_i = \vartheta(\tilde \varphi_i(\B^2\times \S^1\times\{x_0\}^{n-3}))$ for $i=1,2$. Let now $\theta = (\theta_1,\theta_2,\theta_3,\theta_4) \colon C\cap \R^3\to \R^4$ be Semmes's re-embedding (see \cite[Definition 3.2]{SemmesS:Goomsw}) unlinking $t_1$ and $t_2$ in $\R^4$. The diffeomorphism $G_1$ is now obtained by extending the embedding $C\cap \R^3\to \R^{n+1}$, $x\mapsto (\theta_1(x),\theta_2(x),\theta_3(x),x_4,\ldots, x_n, \theta_4(x))$. We leave the technical details to the interested reader and merely refer to Semmes's isotopy extension lemma \cite[Lemma 4.1]{SemmesS:Goomsw}. 

The construction of the diffeomorphism can now be iterated in $G_1(C_1\cup C_2)$ to obtain, for each $k$, a diffeomorphism $G_k\colon U \to U$ satisfying
\[
G_k|\left(U \setminus \bigcup_{|w|=k-1}C_w\right) = G_{k-1}|\left(U\setminus \bigcup_{|w|=k-1} C_w\right),
\]
and for which, for each $w\in \cW_2$ of length $k$, holds that
\begin{itemize}
\item[(a)] $G_k|C_w$ is a $\lambda$-similarity, 
\item[(b)] $G_k(\tilde \varphi'_w(x),\lambda^k t) = g_w(x,t)$ for $(x,t)\in C$, and 
\item[(c)] $G_k(\tilde\varphi'_w(x),0) \subset \R^n\times \{\lambda_k\}$, where $\lambda_k = \sum_{i=0}^k \lambda^i$. 
\end{itemize}
We refer to \cite[Lemma 3.21]{SemmesS:Goomsw}, or \cite[Section 6]{PW}, for more details.
\end{proof}

\subsection{The space $\S^n/\mathrm{BB}$ is an $n$-sphere}
\label{sec:Bing}

We show now that $\S^n/\mathrm{BB}$ is homeomorphic to $\S^n$;  see e.g. DeGryse--Osborne \cite{DeGryseD:WilCs}. For the purposes of our main theorem (Theorem \ref{thm:point}), we emphasize the smoothness properties of this homeomorphism and formulate the result as follows. 

\begin{prop}
\label{prop:BB_Bing}
There exists a map $\hat \rho \colon \S^n \to \S^n$ which restricts to a diffeomorphism $\hat \rho|\left(\S^n\setminus \cS(\fI_{B,n})\right) \colon \S^n\setminus \cS(\fI_{B,n}) \to \S^n \setminus \hat \rho(\cS(\fI_{B,n}))$ and for which there exists a homeomorphism $\S^n/\mathrm{BB}\to \S^n$ so that the diagram 
\[
\xymatrix{
\S^n \ar[rr]^{\hat\rho} \ar[dr]^{x\mapsto [x]} && \S^n \\
& \S^n/\mathrm{BB} \ar[ur]^{\approx} & }
\]
commutes. In particular, $\hat \rho(\cS(\fI_{B,n}))$ is a Cantor set.
\end{prop}

The proof of Proposition \ref{prop:BB_Bing} may be considered as classical. In the heart of the argument is Bing's original shrinking lemma \cite[Lemma, p.358]{BingR:Homb3s}, which we formulate as follows. 

\begin{lem}
\label{lemma:Bing_shrink}
Let $\fI_{B}=(\B^2\times \S^1;\varphi_1,\varphi_2)$ be an initial package for the Bing double, $k\in\N$, and $\varepsilon>0$. Then there exists an integer $m\geq k$ and a diffeomorphism $\hat \varrho \colon \B^2\times \S^1 \to \B^2\times \S^1$ for which 
\begin{itemize}
\item[(1)] $\hat \varrho|(\B^2\times \S^1) \setminus \bigcup_{|w|=k} \varphi_w(\B^2\times \S^1) = \id$, 
\item[(2)] for each word $w$ of length $k$, there exists a neighborhood $\omega_w$ of $\partial \varphi_w(\B^2\times \S^1)$ in $\varphi_w(\B^2\times \S^1)$ for which $\hat \varrho|\omega_w =\id$, and
\item[(3)] for each word $w$ of length $m$, $\diam{\hat \varrho(\varphi_w(\B^2\times \S^1))}<\varepsilon$.
\end{itemize}
\end{lem}

We adapt the proof of Bing's shrinking lemma to obtain a corresponding shrinking lemma for the Bing--Blankinship construction. 

\begin{lem}
\label{lemma:BB_shrink}
Let $n\ge 3$ and let $\fI_{B,n} = (\T;\tilde \varphi_1,\tilde \varphi_2)$ be an initial package for the $n$-dimensional Bing--Blankinship construction, where $\T=\B^2\times (\S^1)^{n-2}$. Let also $k\in \N$, and $\varepsilon>0$. Then there exists an integer $m\geq k$ and a diffeomorphism $\hat \varrho \colon \T \to \T$ so that 
\begin{itemize}
\item[(1)] $\hat \varrho|\T\setminus \bigcup_{|w|=k} \T_w = \id$, 
\item[(2)] for each word $w$ of length $k$, there exists a neighborhood $\Omega_w$ of $\partial \T_w$ in $\T_w$ for which $\hat \varrho|\Omega_w =\id$, and 
\item[(3)] for each word $w$ of length $m$, $\diam{\hat \varrho(\T_w)} < \varepsilon$.
\end{itemize}
\end{lem}

\subsubsection{Bing's shrinking rearrangement; Proof of Lemma \ref{lemma:Bing_shrink}}

As a preparation for the proof of Lemma \ref{lemma:BB_shrink}, we recall in this section Bing's shrinking argument for the decomposition $G_{\fI_B}$ in \cite{BingR:Homb3s}. Let $t\subset \R^3$ be a solid $3$-torus and $(t;\phi_1,\phi_2)$ an initial package equivalent to $\fI_B$. 

Let $\varepsilon>0$. It suffices to show that there exists $m\in \N$ and a diffeomorphism $h \colon t\to t$ for which there exists a neighborhood $\omega \subset t$ of $\partial t$ so that $h|\omega = \id$ and $\diam{h(t_w)}<\varepsilon$ for each $w\in \cW_2$ of length $m$. This is the case $k=0$ of the statement. Since the diffeomorphism  $\phi_w \colon t\to t_w$ is absolutely continuous for each $w\in \cW_2$, the general case follows from this special case.

\medskip

\emph{Step 0}. 
Let $\tau_0 = t$. We fix an initial package $\fI'_1=(t;\phi'_1,\phi'_2)$ equivalent to $\fI_B$ so that the solid torus $\phi'_i(t)$ is a tubular $\delta_0$-neighborhood of a smooth curve $\sigma_i \subset t$ for some $\delta_0\in(0,\varepsilon/4)$ for both $i=1,2$. Recall that a smooth curve $\sigma\subset t$ is an image of a smooth embedding $\S^1\to t$ and a $\delta$-neighborhood of $\sigma$ is $B^3(\sigma,\delta) = \{ x\in \R^3 \colon \dist(x,\sigma)<\delta\}$. Let $h_0 \colon t\to t$ be a diffeomorphism conjugating $(t;\phi_1,\phi_2)$ to $\fI'_1$. Note that, by this choice of a new initial package, we merely shrink the width of the rings. In what follows we consider only the solid torus $\tau_1=\phi_1(t)$; the same argument applies verbatim to $\tau_2=\phi_2(t)$ and $\sigma_2$.

Let $m\in\N$ be such that there exists points $x_1,\dots,x_{2m}$ in $\sigma_1$ in a cyclic order so that, for each $x\in \sigma_1$, $\dist(x,\{x_1,\ldots, x_{2m}\})<\varepsilon/4$. For each $i=1,\dots,2m$, let $P_i$ be a $2$-dimensional plane in $\R^3$ meeting $\sigma_1$ at $x_i$ orthogonally and let $D_i=B^3(x_i,\delta_0)\cap P_i \subset \tau_1$. The number $\d_0>0$ can be chosen small enough so that each $D_i$ intersects $\sigma_1$ only at $x_i$. Note that, if a connected set $E\subset \tau_1$ intersects only one of the disks $D_i$, then $\diam{E} < \varepsilon$.

We now follow Bing's argument and show that there exists a diffeomorphism $h \colon t \to t$ so that, for each $w\in \cW_2$ of length $m$, $h(t_w)$ intersects only one of the disks $D_i$ for $i=1,\dots, 2m$. This concludes the proof. 

\medskip

\emph{Step 1}. Let $\delta_1<\delta_0$ and let $\fI'_2=(\tau_1;\phi'_{11},\phi'_{12})$ be an initial package equivalent to $\fI_B$ satisfying the following properties. For $i=1,2$, we assume that the solid torus $\tau_{1i} = \phi'_{1i}(\tau_1)$ is a $\delta_1$-neighborhood of a smooth curve $\sigma_{1i}\subset \tau_1$. We also assume that 
\[
\begin{cases}
\#(\sigma_{11} \cap D_j) &= 2\\
\#(\sigma_{12} \cap D_j) &= 0
\end{cases} 
\quad \mathrm{and} \quad 
\begin{cases}
\#(\sigma_{11} \cap D_{m+j}) &= 0 \\
\#(\sigma_{12} \cap D_{m+j}) &= 2
\end{cases}
\]
for $j=1,\ldots, m$. Since $\fI'_2$ is equivalent to $\fI_B$, these conditions force $\sigma_{11}$ and $\sigma_{12}$ to link between disks $D_1$ and $D_{2m}$ and between $D_m$ and $D_{m+1}$; see Figure \ref{fig:BingStep1} for the configuration. 

\begin{figure}[h!]
\includegraphics[scale=0.40]{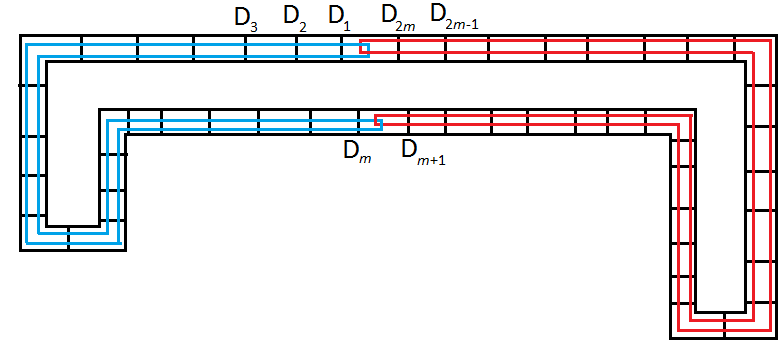}
\caption{The tubes $\tau_{11}$ (blue) and $\tau_{12}$ (red) inside $\tau_1$.}
\label{fig:BingStep1}
\end{figure}

Due to the equivalence of initial packages, there exists a diffeomorphism $h_1 \colon t\to t$ which is the identity in a neighborhood of $t\setminus (\tau_1\cup \tau_2)$ and satisfies $h_1(\phi_{1j}(t)) = \phi'_{1j}(\tau_1)$ for $j=1,2$. Similar rearrangement is possible in the solid torus $\tau_2$. 

If $m=1$, then $h_1$ satisfies the condition $\diam{h_1(\phi_w(t))}<\varepsilon$ for all words $w$ of length $k=2$. 

\smallskip

\emph{Step 2}. 
Suppose that $m\geq 2$. We focus on the solid torus $\tau_{11}=\phi'_{11}(\tau_1)$; constructions in other tori $\tau_{ij} = \phi'_{ij}(\tau_i)$ are verbatim. Following the idea in Step 1, we fix smooth curves $\sigma_{111}$ and $\sigma_{112}$ in $\tau_{11}$ as in Figure \ref{fig:BingStep2} linked between the disks $D_1$ and $D_2$ and the other between $D_{m-1}$ and $D_m$. We also require that $\sigma_{111}$ intersects exactly $D_1,\dots,D_{m-1}$ and $\sigma_{112}$ intersects exactly $D_2,\dots,D_{m}$. Furthermore, if $m\geq 3$, for each $i=1,2$, we require that $\sigma_{11i}$ intersects each of $D_2,\dots, D_{m-1}$ exactly at $2$ points. 

\begin{figure}[h!]
\includegraphics[scale=0.40]{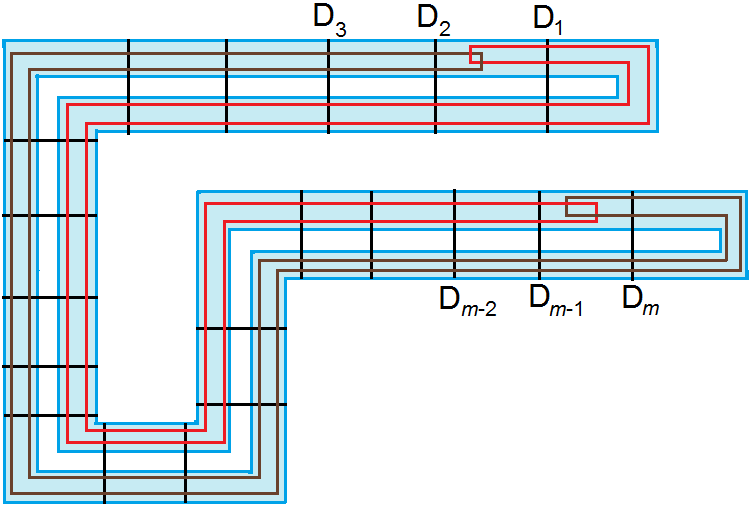}
\caption{The tubes $\tau_{111}$ (brown) and $\tau_{112}$ (red) inside $\tau_{11}$.}
\label{fig:BingStep2}
\end{figure}

As in Step 1, we fix $\delta_2\in (0,\delta_1)$ and an initial package $\fI'_3 = (\tau_{11};\phi'_{111},\phi'_{112})$ so that $\phi'_{11i}(\tau_{11}) = B^3(\sigma_{11i},\delta_2)$ for $i=1,2$. Similarly with $\sigma_{111},\sigma_{112}$, we may assume that $\tau_{111}=\phi'_{111}(\tau_{11})$ intersects exactly $D_1,\dots,D_{m-1}$ and $\tau_{112}=\phi'_{112}(\tau_{11})$ intersects exactly $D_{2},\dots,D_{m}$. 

There exists now a diffeomorphism $h_2 \colon t\to t$ which coincides with $h_1$ on $t\setminus \bigcup_{|w|=2} \tau_w$. Moreover, if $m=2$, $\diam{h(\phi_w(t))} < \varepsilon$ for $|w|=3$. %

\medskip

\emph{Induction assumption.} Suppose we have continued the process for $k\ge 1$ steps and that there exists $\delta_k>0$ so that, for each $w\in \cW_2$ of length $k+1$, 
\begin{itemize}
\item[(a)] the solid torus $\tau_{w}$ is the $\delta_k$-neighborhood of the smooth core curve $\sigma_w$ of $\tau_w$, 
\item[(b)] $\sigma_{w}$ intersects exactly $m-k+1$ consecutive disks in $\{D_1,\ldots, D_{2m}\}$, 
\item[(c)] $\sigma_{w}$ intersects $m-k$ consecutive disks in $\{D_1,\ldots, D_{2m}\}$ at exactly two points.
\end{itemize}

\medskip

\emph{Step $k+1$.} 
Let $w$ be a word of length $k+1$ and let $j\in \N$ be such that $\sigma_w$ intersects disks $D_{j},\dots, D_{j+m-k}$. By the previous step, we may assume that $\sigma_{w}$ intersects disks $D_{j+1},\ldots, D_{j+m-k}$ at exactly two points.

Let $\sigma_{w 1}$ and $\sigma_{w2}$ be smooth pair-wise disjoint curves in $\tau_w$ with the following properties:
\begin{itemize}
\item[(1)] $\sigma_{w1}$ and $\sigma_{w2}$ link as in Figure \ref{fig:Bing}, that is, $(\tau_w, \sigma_{w1}\cup \sigma_{w2})$ and $(\tau_0,\sigma_1\cup \sigma_2)$ are diffeomorphic as pairs;
\item[(2)] $\sigma_{w1}$ and $\sigma_{w2}$ link between $D_j$ and $D_{j+1}$ and between $D_{j+m-k-1}$ and $D_{j+m-k}$; 
\item[(3)] there exists $j\in \{1,\ldots, 2m\}$ so that $\sigma_{w1}$ intersects exactly disks $D_j,\dots,D_{j+m-k-1}$ and $\sigma_{w2}$ exactly disks $D_{j+1},\dots,D_{j+m-k}$; and  
\item[(4)] if $m\geq k+2$, $\sigma_{wi}$ intersects each disk $D_{j+1},\dots D_{j+m-k-1}$ exactly at $2$ points for $i=1,2$.
\end{itemize}

Let now $\delta_{k+1} \in (0,\delta_k)$ be such that there exists smooth embeddings $\phi'_{wi}\colon \tau_w \to \tau_w$ for which $\phi'_{wi}(\tau_w)= B^3(\sigma_{wi},\delta_{k+1})$ and $(\tau_{w}; \phi'_{w1},\phi'_{w2})$ is equivalent to the initial package $\fI_B$. 

We conclude again that there exists a diffeomorphism $h_{k+1}\colon t\to t$ which agrees with $h_k$ in $t\setminus \bigcup_{|v|=k+1} \tau_v$ and satisfies $h_{k+1}(\phi_{wi}(t)) = \phi'_{wi}(\tau_w)$ for each $|w|=k$ and $i=1,2$. This concludes the induction step.

\medskip

\emph{Step $m$; end of the process.}
Suppose that we have reached the step $k=m$. Then, for each $w\in \cW_2$ of length $m+1$, the solid torus $\tau_w$ intersects exactly one of the disks $D_1,\ldots, D_{2m}$. Thus $\diam{h_{m}(t_w)} < \varepsilon$ for $|w|=m+1$. This completes the proof.

\subsubsection{Proof of Lemma \ref{lemma:BB_shrink}}

We apply Bing's idea to rearrange the Blankinship rings $\T_w$ (or equivalently rings $\vartheta(\T_w)$) so that the diameters decrease to zero as $|w|\to \infty$. We discuss the details of the proof of Lemma \ref{lemma:BB_shrink} only in dimension $n=4$ for brevity. The homeomorphism $\hat \varrho$ is obtained similarly in higher dimensions. 

We use notation and constructions from the proof of Lemma \ref{lemma:Bing_shrink}. In the forthcoming steps, we endow each curve $\sigma \subset \R^3$ with the restriction of the Euclidean metric and denote this metric by $d_{\sigma}$. We also consider product spaces $(\sigma,d_{\sigma})\times \B^2(\delta)$ and $(\sigma,d_{\sigma}) \times \B^2(\delta) \times (\hat{\sigma},d_{\hat{\sigma}})$ for $\delta>0$, which we endow with the $\ell^2$-metrics
\[
d_{\sigma,\delta}((x,z),(x',z')) = \left( d_{\sigma}(x,x')^2 + |z-z'|^2\right)^{1/2}
\]
and
\[
d_{\sigma,\delta,\hat{\sigma}}((x,z,y),(x',z',y')) = \left( d_{\sigma}(x,x')^2 + |z-z'|^2 + d_{\hat{\sigma}}(y,y')^2\right)^{1/2},
\]
respectively, where $(x,z),(x',z')\in (\sigma,d_{\sigma})\times \B^2(\delta)$ and $(x,z,y),(x',z',y')\in  (\sigma,d_{\sigma}) \times \B^2(\delta) \times (\hat{\sigma},d_{\hat{\sigma}})$. 

We begin with a simple observation on bilipschitz parametrizations of \emph{tubular neighborhoods} of smooth curves in $\R^3$, which we record as a lemma; see \cite[Theorem 9.20]{Spivak}. Recall that 
\[ B^3(\sigma,\d) = \{x\in\R^3 \colon \dist(x,\sigma) \leq \d\}\] 
denotes a $\delta$-neighborhood of $\sigma$ in $\R^3$. 

\begin{lem}\label{lemma:TNT}
Suppose $\sigma \subset \R^n$ is a closed smooth simple curve. Then, for each $L>1$, there exists $\delta>0$ and an $L$-bilipschitz diffeomorphism 
\[ 
e_{\sigma,\delta}\colon (\sigma,d_{\sigma})\times\B^{2}(\delta)\to B^3(\sigma,\delta)
\] 
such that $e_{\sigma,\delta}(x,0) = x$ for all $x\in\sigma$.
\end{lem}

Let $\sigma \subset \R^3$ be a smooth simple closed curve and $\delta>0$. Given a diffeomorphism $e_{\sigma,\delta}$ as in Lemma \ref{lemma:TNT}, we define a diffeomorphism $\hat{e}_{\sigma,\d} \colon \B^{2}(\d) \times (\sigma,d_{\sigma}) \to B^3(\sigma,\d)$ by $\hat{e}_{\sigma,\d}(u,x) = e_{\sigma,\d}(x,u)$. If $\hat \sigma$ is another smooth simple closed curve in $\R^3$, we call a diffeomorphism
\[ 
G_{\sigma,\d,\hat{\sigma}} \colon B^3(\sigma,\d) \times (\hat{\sigma},d_{\hat{\sigma}}) \to (\sigma,d_{\sigma}) \times B^3(\hat{\sigma},\d)
\]
defined by
\[
G_{\sigma,\d,\hat{\sigma}} = (\id_{\sigma}\times \hat{e}_{\hat{\sigma},\d}) \circ (e_{\sigma,\d}\times\id_{\hat{\sigma}})^{-1}
\]
a \emph{neighborhood switch}.

To simplify the notation in the proof of Lemma \ref{lemma:BB_shrink}, we define, for each word $w=i_1\cdots i_k$ and $\hat w = j_1\cdots j_\ell$ in $\cW_2$, their \emph{interlace} $[w,\hat w]\in \cW_2$ by 
\[
[w,\hat w] = \left\{ \begin{array}{ll}
i_1 j_1 i_2 j_2 \cdots i_k j_k j_{k+1}\cdots j_\ell, & \mathrm{if}\ \ell \ge k, \\
i_1 j_1 i_2 j_2 \cdots i_\ell j_\ell i_{\ell+1} \cdots i_k, & \mathrm{if} \ell < k. 
\end{array}\right.
\]

\begin{proof}[Proof of Lemma \ref{lemma:BB_shrink}]
Let $\varepsilon >0$. Since the diffeomorphism  $\phi_w \colon t\to t_w$ is absolutely continuous for each $w\in \cW_2$, it suffices to consider only the case $k=0$. In what follows, we denote $\T = \S^1\times\B^2\times\S^1$ and $\T_w = \tilde{\varphi}_w(\T)$ for $w\in\cW_2$. For completion, we set $\sigma_0 = \hat{\sigma}_0 = \S^1 \subset \R^3$.

We construct iteratively a finite ordered tree $(T_u)_{|u|\le m}$ equivalent to $\fI_{B,n}$, where $u\in \cW_2$ has length at most $2m$ and $\diam T_u< \varepsilon$ for each $|u|=2m$. The required $4$-tubes $T_u$ and the level $m$ are found by showing that there exists $\delta_0>0$ so that, if $m\in \N$ is the level associated to diameter $\varepsilon' = \varepsilon \delta_0/20$ in Lemma \ref{lemma:Bing_shrink}, we find $4$-tubes $T_u$ for which there exists $(\delta_0/8)$-bilipschitz diffeomorphisms $T_{[w,\hat w]} \to (\sigma_w,d_{\sigma_w}) \times \B^2(\delta_m) \times (\sigma_{\hat w}, d_{\sigma_{\hat w}})$ for each $w,\hat w\in \cW_2$ of length $m$, where $\sigma_w$ and $\sigma_{\hat w}$ are core curves of $3$-tubes $\tau_w$ and $\tau_{\hat w}$, respectively, as in the proof of Lemma \ref{lemma:Bing_shrink}. The diameter bound follows then from diameter bounds of curves $\sigma_w$.

Roughly speaking, the curves $\sigma_w$ and $\sigma_{\hat{w}}$ are obtained by applying Bing's shrinking rearrangement in the two $\S^1$ directions of $\mathbb{T}$. In Steps $(k,1)$ below, $k\in\N\cup\{0\}$, we apply the shrinking rearrangement for the first direction and Steps $(k,2)$ for the second.

Fix $\d_0 \in (0,1)$ small enough so that the diffeomorphism $e_{\sigma_0,\d_0}$ of Lemma \ref{lemma:TNT} is $2$-bilipschitz. Set $\tau_0 =  B^3(\sigma_0,\d_0)$ and define $R_0 \colon \T \to \tau_0 \times\hat{\sigma}_0$ by 
\[ 
(x,u,y) \mapsto (e_{\sigma_0,\d_0}(x,\d_0u),y).
\] 
Then $R_0$ is a $(2/\d_0)$-bilipschitz diffeomorphism.

\medskip

\emph{Step (0,1)}. Let $\tau_1 = B^3(\sigma_1,\d_1)$ and $\tau_2 = B^3(\sigma_2,\d_1)$ be $3$-tubes in $\tau_0$, where $\sigma_1$ and $\sigma_2$ are smooth simple curves linked in $\tau_0$ as in Figure \ref{fig:Bing}, and $\d_1 < \varepsilon \delta_0/20$. We choose $\d_1$ small enough so that the diffeomorphisms $e_{\sigma_{1},\d_1}$ and $e_{\sigma_{2},\d_1}$, in Lemma \ref{lemma:TNT}, and the diffeomorphism $\hat{e}_{\hat{\sigma}_0,\d_1}$, associated to $e_{\hat \sigma,\delta_1}$, are $2^{1/2}$-bilipschitz. For each $i=1,2$, let $\phi_{i} \colon \tau_0\to\tau_0$  be a diffeomorphism with $\phi_{i}(\tau_0) = \tau_i$. Define $\Phi_i \colon \T \to \T$ by $\Phi_i = R_0^{-1}\circ (\phi_{i}\times \id) \circ R_0$ and set $T_i = \Phi_i(\T) \subset \T$. Note that the initial package $(\T;\Phi_1,\Phi_2)$ is conjugate to $(\T;\tilde\varphi_1,\tilde\varphi_2)$ and that the diagram
\[\xymatrix{
\T \ar[d]^{R_0} \ar[r]^{\Phi_{i}}
& T_{i}
\\
\tau_0\times(\S^1,d_{\S^1}) \ar[r]^{\phi_{i}\times \id} 
&\tau_i\times(\S^1,d_{\S^1}) \ar[u]_{R_0^{-1}|\tau_i\times \S^1}
}\]
commutes.

\medskip

\emph{Step (0,2)}. Set $\hat{\tau}_0 = B^3(\hat{\sigma}_0,\d_1)$ and, for $i=1,2$, let $R_{i} \colon T_{i} \to (\sigma_{i},d_{\sigma_{i}})\times \hat{\tau}_0$ be the mapping
\[
R_{i} =  G_{\sigma_i,\d_1,\hat{\sigma}_0}\circ R_0.
\]
Let $\hat{\tau}_1 = B^3(\hat{\sigma}_1,\hat{\d}_1)$ and $\hat{\tau}_2 = B^3(\hat{\sigma}_2,\hat{\d}_1)$ be $3$-tubes in $\hat{\tau}_0$, where $\hat{\sigma}_1$ and $\hat{\sigma}_2$ are smooth simple curves, linked inside $\hat{\tau}_0$ as in Figure \ref{fig:Bing}, and $\hat{\d}_1 < \d_1$. Moreover, we may assume that $\hat{\d}_1$ is small enough so that the diffeomorphisms $\hat{e}_{\hat{\sigma}_{j},\hat{\d}_1}$ and $e_{\sigma_{i},\hat{\d}_1}$ are $2^{1/4}$-bilipschitz for $i,j\in\{1,2\}$. For each $j=1,2$, let $\hat{\phi}_{j} \colon \hat{\tau}_0 \to \hat{\tau}_0$ be a diffeomorphism satisfying $\hat{\phi}_{j}(\tau_0) = \hat{\tau}_{j}$. Define also $\Phi_{ij} \colon T_i \to T_i$ by $\Phi_{ij} = R_{i}^{-1}\circ (\id\times\phi_{j}) \circ R_i$ and set $T_{ij} = \Phi_{ij}(T_i) \subset T_{i}$. Then the initial package $(T_i;\Phi_{i1},\Phi_{i2})$ is conjugate to $(\T;\tilde \varphi_1,\tilde \varphi_2)$ and the diagram
\[
\xymatrix{
T_{i} \ar[d]^{R_{i}} \ar[r]^{\Phi_{ij}}
& T_{ij}
\\
(\sigma_{i},d_{\sigma_{i}})\times \hat{\tau}_0 \ar[r]^{\id\times\hat{\phi}_{j}} 
&(\sigma_{i},d_{\sigma_{i}})\times \hat{\tau}_j \ar[u]_{R_{i}^{-1}|\sigma_{i}\times \hat \tau_j}}
\]
commutes.

\medskip

Let $m\in\N$ be the number of steps needed in Bing's shrinking procedure for $\tau_0$ and $\varepsilon\d_0/20$. We may assume that the same number of $m$ steps is needed in Bing's shrinking procedure for the simple curve $\hat{\tau}_0$ and diameter $\varepsilon\d_0/20$. We now apply Steps $1$ to $m$ of Bing's shrinking procedure to both curves $\tau_0$ and $\hat \tau_0$. 

\medskip

Suppose that after Step $(k,1)$ and Step $(k,2)$ we have obtained radii 
\[
\delta_0 > \hat \delta_0 > \cdots > \delta_k > \hat \delta_k >0
\]
and, for each $\ell \in \{1,\ldots, k\}$, $w=i_1\cdots i_\ell$ and $\hat w = j_1\cdots j_{\ell-1}$ in $\cW_2$, and $j=1,2$, we have
\begin{enumerate}
\item smooth closed simple curves $\sigma_w$ and $\hat \sigma_{\hat w j}$ as in the Step $\ell$ of Bing's shrinking procedure so that $\tau_w = B^3(\sigma_w,\d_\ell)$ and $\hat{\tau}_{\hat w j} = B^3(\hat{\sigma}_{\hat w j},\hat{\d}_\ell)$ are $3$-tubes;
\item a $4$-tube $T_{[w,\hat w j]}$ and a diffeomorphic embedding
\[ 
\Phi_{[w,\hat w j]} \colon T_{[w,\hat w]} \to T_{[w,\hat w]}
\] 
for which $T_{[w,\hat w j]} = \Phi_{[w,\hat w j]}(T_{[w,\hat w]})$ and $(T_{[w,\hat w]}; \Phi_{[w,\hat w 1]}, \Phi_{[w,\hat w 2]})$ is conjugate to $(\T; \tilde \varphi_1, \tilde \varphi_2)$; 
\item diffeomorphisms $e_{\sigma_w,\hat{\d}_{k}}$ and $\hat{e}_{\hat{\sigma}_{\hat w j,\hat{\d}_{k}}}$ as in Lemma \ref{lemma:TNT} which are $2^{2^{-2k}}$-bilipschitz; and 
\item a diffeomorphism $R_{[w,\hat w]} \colon T_{[w,\hat w j]} \to (\sigma_{w},d_{\sigma_{w}})\times \hat{\tau}_{\hat w}$ satisfying
\[
R_{[w,\hat w]} = G_{\sigma_{w},\d_{k},\hat{\sigma}_{\hat w}} \circ R_{[w,\hat w j]}.
\]
\end{enumerate}

\medskip 

\emph{Step $(k+1,1)$}.
Let $w,\hat w\in \cW_2$ be words of length $|w|=k$ and $|\hat w|=k-1$. We define
\[ 
R_{[w,\hat wj]} \colon T_{[w,\hat wj]} \to \tau_{w} \times (\hat{\sigma}_{\hat wj},d_{\hat{\sigma}_{\hat w j}})
\] 
for $j=1,2$ by
\[
R_{[w,\hat w j]} = (G_{\sigma_{w}, \hat{\d}_{k},\hat{\sigma}_{\hat w j}})^{-1} \circ R_{[w,\hat w]}.
\]
Let $\sigma_{w1}$ and $\sigma_{w2}$ be smooth simple closed curves in $\tau_{w}$ as in Step $k+1$ of Bing's shrinking procedure, and $\d_{k+1} <\hat{\d}_{k}$ be small enough so that the diffeomorphisms $e_{\sigma_{wi},\d_{k+1}}$ and $\hat{e}_{\hat{\sigma}_{\hat w j},\d_{k+1}}$ of Lemma \ref{lemma:TNT} are $2^{2^{-2k-1}}$-bilipschitz for $i,j\in \{1,2\}$. 

For $i=1,2$, let $\tau_{w i} = B^3(\sigma_{wi},\d_{k+1})$ and $\phi_{wi} \colon \tau_{w} \to \tau_{wi}$ be a diffeomorphism. Define also, for $j=1,2$,
\[ 
\Phi_{[wi,\hat wj]} \colon T_{[w,\hat wj]} \to T_{[w,\hat wj]}
\]
by
\[ 
\Phi_{[wi,\hat wj]} = R_{[w,\hat wj]}^{-1}\circ (\phi_{wi}\times\id)\circ R_{[w,\hat w j]},
\] 
and set
\[ 
T_{[wi,\hat wj]} = \Phi_{[wi,\hat wj]}(T_{[w,\hat wj]}) \subset T_{[w,\hat wj]}.
\] 

\medskip

\emph{Step $(k+1,2)$}.
For $w$ and $\hat w$ in $\cW_2$ of length $k$, define
\[
R_{[wi,\hat w]} \colon T_{[wi,\hat w]} \to (\sigma_{wi},d_{\sigma_{wi}}) \times \hat{\tau}_{\hat w}\] 
for $i=1,2$,
by
\[
R_{[wi,\hat w]} = G_{\sigma_{wi}, \d_{k+1},\hat{\sigma}_{\hat w}} \circ R_{[w,\hat w]}.
\]
Let $\sigma_{\hat w 1}$ and $\sigma_{\hat w 2}$ be two smooth simple closed curves in $\hat{\tau}_{\hat w}$ as in Step $k+1$ of Bing's shrinking procedure, and $\hat{\d}_{k+1} <\d_{k+1}$ be small enough so that the diffeomorphisms $e_{\sigma_{wi},\hat{\d}_{k+1}}$ and $e_{\hat{\sigma}_{\hat w j},\hat{\d}_{k+1}}$ of Lemma \ref{lemma:TNT} are $2^{2^{-2k-2}}$-bilipschitz for $i=1,2$. 

For each $j\in\{1,2\}$, let $\hat{\tau}_{\hat w j} = B^3(\hat{\sigma}_{\hat w j},\hat{\d}_{k+1})$ and $\hat{\phi}_{\hat w j} \colon \hat{\tau}_{\hat w} \to \hat{\tau}_{\hat w j}$ be a diffeomorphism. Define now 
\[ 
\Phi_{[wi, \hat wj]} \colon T_{[wi, \hat w]} \to T_{[wi, \hat w]} 
\]
for $i,j\in \{1,2\}$ by
\[
\Phi_{[wi, \hat wj]} = R_{[wi,\hat w]}^{-1}\circ (\id\times\hat{\phi}_{wi}) \circ R_{[wi,\hat wj]} 
\] 
and set
\[ 
T_{[wi,\hat wj]} = \Phi_{[wi,\hat wj]}(T_{[wi,\hat w]}) \subset T_{[wi,\hat w]}
\] 
for each $i$ and $j$.

\medskip

Suppose we have completed Step $(m,2)$ and let $w=i_1\cdots i_m$ and $\hat w = j_1\cdots j_m$ be words in $\cW_2$. By the choice of radii $\delta_\ell$ and $\hat \delta_\ell$ for $\ell = 1,\ldots, m$, the diffeomorphism 
\[
R_{[w,j_1\cdots j_{m-1}]} \colon T_{[w,\hat w]} \to B^3(\sigma_{w},\d_m) \times (\hat{\sigma}_{\hat w}, d_{\hat{\sigma}_{\hat w}}),
\]
is $(8/\delta_0)$-bilipschitz. Indeed, denote, for each $\ell \in \{1,\ldots,m\}$,$w_\ell = i_1\cdots i_\ell$ and $\hat w_\ell = j_1\cdots j_\ell$. Then
\begin{eqnarray*}
R_{[w,\hat w_{m-1}]} &=& G_{\sigma_{w}, \d_{m},\hat{\sigma}_{\hat w_{m-1}}} \circ R_{[w_{m-1},\hat w_{m-1}]} \\
&=& G_{\sigma_{w}, \d_{m},\hat{\sigma}_{\hat w_{m-1}}} \circ (G_{\sigma_{w_{m-1}}, \hat{\d}_{m-1},\hat{\sigma}_{\hat w_{m-1}}})^{-1}\circ R_{[w_{m-1},\hat w_{m-2}]} \\
&=& G'_m \circ \cdots \circ G'_2 \circ G_{\sigma_{i_1}, \delta_1, \hat \sigma_0} \circ R_0,
\end{eqnarray*}
where
\[
G'_\ell = G_{\sigma_{w_\ell}, \d_{\ell},\hat{\sigma}_{\hat w_{\ell-1}}} \circ (G_{\sigma_{w_{\ell-1}}, \hat{\d}_{\ell-1},\hat{\sigma}_{\hat w_{\ell-1}}})^{-1}
\]
for each $\ell = 2,\ldots, m$.

Since $\diam{\sigma_{w}} < \varepsilon\d_0/20$ and $\diam{\hat{\sigma}_{\hat w}} < \varepsilon\d_0/20$, 
\[ 
\diam{T_{[w,\hat w]}} \leq 8(\d_0)^{-1}(\diam{\sigma_{w}} + \diam{\hat{\sigma}_{\hat w}} + \d_m) < \varepsilon.
\]
This concludes the proof of Lemma \ref{lemma:BB_shrink}.
\end{proof}

\begin{proof}[{Proof of Proposition \ref{prop:BB_Bing}}]
The proof of Proposition \ref{prop:BB_Bing} is identical to the discussion in \cite[Section 3.III]{BingR:Homb3s}. By Lemma \ref{lemma:BB_shrink}, there exists $k_1 \in \N$ and a diffeomorphism $\hat{\varrho}_1\colon \S^n \to \S^n$ which leaves each point of $\S^n \setminus (T_1\cup T_2)$ fixed and maps each torus $T_w$, for $|w|=k_1$, into a set of diameter less than $1$. Then, by uniform continuity of $(\hat{\varrho}_1)^{-1}$, there exists $\varepsilon>0$ for which a set $E\subset \S^n$ has diameter less than $1/2$ if $\diam{\hat{\varrho}_1(E)}< \varepsilon$. Reapplying Lemma \ref{lemma:BB_shrink}, we find an integer $k_2 > k_1$ and a diffeomorphism $\hat{\varrho}_2' \colon \S^n \to \S^n$ which leaves each point of $\S^n \setminus {\bigcup_{|w|=k_1} T_w}$ fixed and maps each torus $T_w$, for $|w|=k_2$, into a set of diameter less than $\varepsilon$. Then $\hat{\varrho}_2 = \hat{\varrho}_1\circ \hat{\varrho}_2'$ is a diffeomorphism of $\S^n$ into itself which maps each torus $T_w$, for $|w|=k_1$, to a set of diameter $1$ and each torus $T_w$, for $|w|=k_2$, to a set of diameter $1/2$.

Iterating this procedure, we find a sequence of diffeomorphisms $\hat{\varrho}_1,\hat{\varrho}_2,\dots$ and integers $k_1 < k_2 < \cdots$ for which $\hat{\varrho}_{m+1}|\S^n\setminus \bigcup_{|w|=k_m} \T_w = \hat{\varrho}_m|\S^n\setminus \bigcup_{|w|=k_m} \T_w$, and $\diam{\hat{\varrho}_m(\T_w)} < 1/m$ for each $|w|=k_m$. The limit $\hat \rho = \lim_{m\to\infty} \hat{\varrho}_m$ is a map $\hat \rho \colon \S^n \to \S^n$ for which $\hat \rho(\cS(\fI_{B,n}))$ is a Cantor set and the restriction $\hat \rho|S^n \setminus \cS(\fI_{B,n}) \colon \S^n \setminus \cS(\fI_{B,n}) \to \hat \rho(\S^n)\setminus \hat \rho(\cS(\fI_{B,n})$ is a diffeomorphism. The proof is complete.
\end{proof}

\subsection{An almost smooth metric on $\S^n$ associated to $\S^n/\mathrm{BB}$}
\label{sec:almost_smooth_metric}

As a direct corollary of Propositions \ref{prop:BB_Semmes} and \ref{prop:BB_Bing}, we obtain an almost smooth metric on $\S^n$ having the Bing--Blankinship Cantor set as a singular set.

\begin{cor}
\label{cor:metric_E}
Let $n \ge 3$, $\fI_{B,n}$ be a Bing--Blankinship initial package, and let $\hat \rho \colon \S^n \to \S^n$  be a map as in Proposition \ref{prop:BB_Bing}. Let also $E= \hat \rho(\cS(\fI_{B,n}))$.

Then, there exists a Cantor set $E\subset \S^n$ and a Riemannian metric $g$ in $\S^n\setminus E$ for which the completion of the associated length metric $d$ is Ahlfors $n$-regular and linearly locally contractible.
\end{cor}

\begin{proof}
Let $\lambda \in (0,2^{-1/n})$ and let $\tilde \rho \colon \S^n\to \S^{n+1}$ be the mapping in Proposition \ref{prop:BB_Semmes}. We set $g$ to be the Riemannian metric $(\tilde \rho \circ \hat \rho^{-1}|\S^n\setminus E)^*g_0$, where $g_0$ is the Riemannian metric on $\S^{n+1}$. Let $d$ be the completion of the length metric associated to $g$. 

Let $\rho'\colon \S^n/\mathrm{BB}\to \S^n$ be the homeomorphism in Proposition \ref{prop:BB_Bing} and $d'$ the pull-back metric $d'(x,y) = d((\rho'\circ \vartheta')(x),(\rho'\circ \vartheta')(y))$ on $\T/G_{\fI_{B,n}}$. Then $d'$ is a Semmes metric on $\S^n/\mathrm{BB}$, with respect to the initial package $\fI_{B,n}$, of scaling constant $\lambda$. Thus $(\T/G_{\fI_{B,n}},d')$ is Ahlfors $n$-regular and LLC by Lemmas \ref{lem:Ahlfors} and \ref{lem:LLC}. Thus also $(\mathbb{S}^n,d)$ is Ahlfors $n$-regular and LLC; see e.g.\;\cite[Proposition 7.8]{PW} and \cite[Proposition 7.9]{PW} for details.
\end{proof}

\section{Virtually interior-essential maps}
\label{sec:viemaps}

In this section we prove a Freedman-Skora type result that, for a virtually interior essential map $\Phi \colon (\B^2,\partial \B^2)\to (\T,\partial \T)$, the number of essential intersections $\Phi(\B^2)\cap \bigcup_{|w|=k}\tilde \varphi_w(\T)$ is at least $2^k$. We begin by introducing terminology.

Let $\omega \subset \B^2$ be a compact and connected $2$-manifold with boundary. The smallest $2$-cell $D_\omega$ in $\B^2$ containing $\omega$ is the \emph{hull of $\omega$ in $\B^2$}, that is, $D_\omega$ is the unique $2$-cell in $\B^2$ containing $\omega$ for which $\partial D_\omega$ is a component of $\partial \omega$. We call $\partial D_\omega$ the \emph{outer boundary of $\omega$} and $\partial \omega \setminus \partial D_\omega$ the \emph{inner boundary of $\omega$}.

Let $M$ be an $n$-manifold with boundary. A map of pairs $\Phi \colon (\omega,\partial \omega) \to (M,\partial M)$ is \emph{interior-inessential} if there exists a map $\Phi' \colon \omega \to \partial M$ for which $\Phi'|\partial \omega = \Phi|\partial \omega$. Otherwise, $\Phi$ is \emph{interior-essential}. Further, a map of pairs $\Phi \colon (\omega, \partial \omega)\to (M,\partial M)$ is \emph{virtually interior-essential} if there exists an interior-essential extension $\hat \Phi \colon (D_\omega,\partial D_\omega) \to (M,\partial M)$ of $\Phi$ satisfying $\hat \Phi(D\setminus \omega)\subset \partial M$.

Let $N\subset \mathrm{int}M$ be an $n$-manifold with boundary and $\omega \subset \B^2$ a compact and connected $2$-manifold with boundary. A map $\Phi \colon (\omega,\partial \omega) \to (M,\partial M)$ \emph{intersects $N$ transversely} if $\Phi^{-1}(\partial N)$ is a closed $1$-manifold, i.e. $\Phi^{-1}(\partial N)$ is a finite pair-wise disjoint collection of circles in $\omega$. In particular, components of $\Phi^{-1}(N)$ are compact and connected $2$-manifolds with boundary if $\Phi$ intersects $N$ transversely. Note that each map $(\omega,\partial \omega) \to (M,\partial M)$ is homotopic, relative to the boundary $\partial \omega$, to a map which intersects $N$ transversely.

For a mapping $\Phi \colon \omega \to M$ intersecting $N$ transversely, we denote by $\Omega(\Phi;N)$ the set of all components $\omega'\subset \Phi^{-1}(N)$ for which $\Phi|\omega' \colon (\omega',\partial \omega') \to (N,\partial N)$ is interior essential. A component $\omega'\in \Omega(\Phi;N)$ is \emph{innermost} if $D_{\omega'}\setminus\omega'$ has no element in $\Omega(\Phi;N)$. We emphasize that $\Omega(\Phi;N)$ is a finite set, since $\Phi^{-1}(\partial N)$ has finitely many components.

The main result of this section is the following proposition. Note that, although not explicitly mentioned, we consider a fixed initial package $\fI_{B,n}=(\T;\tilde \varphi_1,\tilde \varphi_2)$ for the Bing-Blankinship decomposition.

\begin{prop}
\label{prop:vie}
Let $\omega\subset \B^2$ be a compact and connected $2$-manifold, and suppose $\Phi \colon (\omega,\partial \omega)\to (\T,\partial \T)$ is a virtually interior-essential map meeting $\T_1\cup \T_2$ transversely. Then there exists at least two virtually interior essential components in $\Omega(\Phi;\T_1\cup \T_2)$.
\end{prop}

Since $(\T_w,\T_{w1}\cup \T_{w2})$ is homeomorphic, as pairs, to $(\T,\T_1\cup \T_2)$, a simple induction argument yields the following corollary.
\begin{cor}
\label{cor:vie}
Let $\Phi \colon (\B^2,\partial \B^2) \to (\T, \partial \T)$ be an interior essential map which meets each $\T_w$, for $w\in \W_2$, transversely. Then
\[
\# \Omega(\Phi; \bigcup_{|w|=k} \T_w) \ge 2^k
\]
for each $k\ge 0$.
\end{cor}

The proof of Proposition \ref{prop:vie} is based on a homological argument as in Freedman and Skora \cite[Lemma 2.5]{FreedmanM:Strags}. Before the proof we make first some well-known preliminary observations. 

\begin{lem} 
\label{lem:link1}
The homomorphism $\pi_1(\partial \T)\to \pi_1(\T\setminus (\T_1\cup \T_2))$, induced by the inclusion $\partial \T \to \T \setminus (\T_1\cup \T_2)$, is a monomorphism. Furthermore, the homomorphism $\pi_1(\partial \T_i)\to \pi_1(\T\setminus \interior(\T_1 \cup \T_2))$, induced by the inclusion $\partial \T_i \to \T\setminus \interior(\T_1\cup \T_2)$, is a monomorphism for $i=1,2$. 
\end{lem}

\begin{proof}
For $n=3$, i.e.\;for the Bing double, see \cite[Lemmas 2.4 and 2.5]{DeGryseD:WilCs}. For $n\ge 3$, it suffices to observe that 
\[
\T = (\B^2\times \S^1)\times (\S^1)^{n-3}
\]
and
\[
\T_i = \varphi_i(\B^2\times \S^1)\times (\S^1)^{n-3},
\]
where $\varphi_i \colon \B^2 \times \S^1 \to \B^2\times \S^1$, for $i=1,2$, is the embedding in the initial package of the Bing double.
\end{proof}

\begin{cor}\label{cor:ie}
Let $\Phi\colon (\B^2,\partial \B^2) \to (\T,\partial \T)$ be an interior essential map. Then $\Phi^{-1}(\T_1\cap \T_2)\ne \emptyset$. Furthermore, $\Omega(\Phi;\T_1\cup \T_2)\ne \emptyset$ if $\Phi$ meets $\T_1\cup \T_2$ transversely.
\end{cor}
\begin{proof}
Suppose $\Omega(\Phi;\T_1\cup \T_2)=\emptyset$. Then there is a map $\Phi' \colon (\B^2,\partial \B^2)\to (\T\setminus (\T_1\cup \T_2), \partial \T)$ for which $\Phi'|\partial \B^2 = \Phi|\partial \B^2$. Thus $\Phi'|\S^1$ is contractible in $\T\setminus (\T_1\cup \T_2)$, and, by Lemma \ref{lem:link1}, contractible in $\partial \T$. This contradicts the interior essentiality of $\Phi$.
\end{proof}

\begin{lem}
\label{lem:ie}
Suppose that $\Phi \colon (\B^2,\partial \B^2) \to (\T,\partial \T)$ meets $\T_1 \cup \T_2$ transversely and let $\omega\in \Omega(\Phi;\T_1\cup \T_2)$ be an innermost component. Then the restriction $\Phi|\omega \colon (\omega,\partial \omega)\to (\T_1\cup \T_2,\partial (\T_1\cup \T_2))$ is virtually interior essential. 
\end{lem}

\begin{proof}
We may assume that $\omega \subset \Phi^{-1}(\T_1)$. Let $D\subset D_\omega\setminus \omega$ be a component and $E_i=\Phi^{-1}(\T_i) \cap \interior D$ for $i=1,2$. Since $\omega$ is innermost interior essential component in $\Omega(\Phi;\T_1\cup \T_2)$, there exists a map $\Phi'_D \colon D \to \T\setminus \interior (\T_1\cup \T_2)$ satisfying $\Phi'_D|D\setminus (E_1\cup E_2) = \Phi|D\setminus (E_1\cup E_2)$ and $\Phi'_D(E_i)\subset \partial \T_i$ for $i=1,2$. We conclude that $\Phi'_D$ contracts $\partial D$ in $\T\setminus \interior (\T_1\cup \T_2)$. 

Note that $\partial \T_1$ and $\partial \T_2$ are bi-collared in $\T$, that is, for each $i=1,2$, there exists an embedding $b_i\colon \partial \mathbb{T}_i \times [-1,1] \to \mathbb{T}$ such that $b(x,0)=x$. Thus, we conclude that $\Phi'_D|\partial D$ contracts in $\partial \T_1$ by Lemma \ref{lem:link1}. In particular, there exists a map $\Phi_\omega \colon D_\omega \to \T_1$ which extends $\Phi|\omega$ and satisfies $\Phi_\omega(D_\omega\setminus \omega)\subset \partial \T_1$. Thus $\Phi|\omega$ is virtually interior essential.
\end{proof}

The proof of Proposition \ref{prop:vie} is based on the observation that there exists at least two innermost interior essential components in $\Omega(\Phi;\T_1\cup \T_2)$. As a first step, we prove that a virtually interior essential map $(\omega,\partial \omega)\to (\T,\partial \T)$ is homologically non-trivial in $H_2(\T,\partial \T;\Z)$ but intersects $\T_1\cup \T_2$ homologically trivially; cf.\;\cite[Lemma 2.5]{FreedmanM:Strags} and \cite[Lemma 7.2]{HW}. We begin with two general observations on the relative homology of $n$-tubes.

\begin{lem}
\label{lem:homology}
The relative homology group $H_2(\T,\partial \T;\Z)$ is infinite cyclic and generated by the class of the map $G \colon \B^2 \to \T$, $x \mapsto (x,y_0,\dots,y_0)$, where $y_0\in \S^1$.
\end{lem}

\begin{proof}
Let $K=\B^2(0,1/2)\times (\S^1)^{n-2}$. Since the inclusion $(\T,\partial \T) \to (\T,\T\setminus K)$ is a homotopy equivalence of pairs, we have (see e.g.\;\cite[Proposition 3.46]{HatcherA:Algt}), that
\[
H_2(\T,\partial \T;\Z) \cong H_2(\T,\T-K;\Z) \cong H^{n-2}(K;\Z) \cong H^{n-2}((\S^1)^{n-2};\Z) \cong \Z.
\]

To show that $H_2(\T,\partial \T;\Z) = \langle [G] \rangle$, let $\pi \colon \T\to \B^2$ be the natural projection $(x,y_1,\ldots, y_{n-2})\mapsto x$. Since $\pi(\partial \T) = \pi(\partial \B^2 \times (\S^1)^{n-2}) = \partial \B^2$, the map $\pi_* \colon H_2(\T,\partial \T, \Z)\to H_2(\B^2,\partial \B^2;\Z)$ is well-defined. Since $\pi \circ G = \id_{\B^2}$, we conclude that $\langle [G]\rangle = H_2(\T,\partial \T;\Z)$.
\end{proof}

\begin{lem}
\label{lem:not_0}
Let $T$ be an $n$-tube, $\omega \subset \B^2$ a compact and connected $2$-manifold with boundary, and let $\Phi \colon (\omega,\partial \omega) \to (T,\partial T)$ be a virtually interior essential map. Then $[\Phi]\ne 0$ in $H_2(T,\partial T;\Z)$.
\end{lem}

\begin{proof}
Since $T$ is homeomorphic to $\T$, it is enough if we show the lemma for $\T$ only. Let $\hat \Phi \colon D_\omega \to \T$ be an extension of $\Phi$ for which $\hat\Phi(D_\omega \setminus \omega)\subset \partial \T$. We may assume, by extending $\hat \Phi$ further if necessary, that $D_\omega = \B^2$.

Then $[\Phi]=[\hat \Phi]$ in $H_2(\T,\partial \T;\Z)$. Let $\iota \colon \partial \T\to \T$ be the inclusion. Then $[\hat \Phi|\partial D_\omega]\in \ker(\iota_* \colon \pi_1(\partial \T)\to \pi_1(\T))$, since $\hat \Phi$ is interior essential; here we tacitly identify $\partial D_\omega$  with $\S^1$. Thus $[\hat \Phi|\partial D_\omega]=[g]^m$ for $m\ne 0$, where $g\colon \S^1 \to \partial \T$, $x\mapsto (x,y_0,\ldots,y_0)$, and $y_0\in\S^1$. Thus we may assume that $\hat \Phi(z) = (z^m,y_0,\ldots, y_0)=g(z^m)$ for $z\in \S^1\subset \mathbb C$.

Let $G\colon \B^2 \to \T$ be as in Lemma \ref{lem:homology}. We claim that $[\hat \Phi] = m[G]$. Indeed, let $G_m \colon \B^2\to \T$ be the map $z\mapsto G(z^m)$ and $\pi \colon \B^2\times \R^{n-2}\to \T$ the universal cover of $\T$. Let also $\tilde \Phi \colon \B^2\to \B^2\times \R^{n-2}$ and $\tilde G_m\colon \B^2\to \B^2\times \R^{n-2}$ be lifts of $\hat{\Phi}$ and $G_m$, respectively, in $\pi$ so that $\tilde \Phi(z)=\tilde G_m(z)$ for every $z\in \S^1$.
 
Since $\B^2\times \R^{n-2}$ is contractible and $\tilde \Phi - \tilde G_m$ is a $2$-cycle, there exists a $3$-chain $\tilde\sigma$ in $\B^2\times \R^{n-2}$ for which $\tilde \Phi - \tilde G_m = \partial \tilde \sigma$. Thus $\hat{\Phi} - \pi\circ \tilde G_m = \partial \pi_\# \tilde \sigma$. Since $[\pi \circ \tilde G_m] = m[G]$, the claim follows.
\end{proof}

\begin{lem}
\label{lem:homology0}
Let $\omega \subset \B^2$ be a compact and connected $2$-manifold with boundary and let $\Phi \colon (\omega,\partial \omega) \to (\T,\partial \T)$ be a virtually interior-essential map meeting $\T_1\cup \T_2$ transversely. Then $[\Phi|\Phi^{-1}(\T_i)] = 0$ in $H_2(\T_i,\partial \T_i;\Z)$ for $i=1,2$.
\end{lem}

\begin{proof}
Since
\[
H_2(\T_1\cup \T_2, \partial (\T_1\cup \T_2);\Z) = H_2(\T_1,\partial \T_1;\Z) \oplus H_2(\T_2,\partial \T_2;\Z)
\]
it suffices to show that $[\Phi|\Phi^{-1}(\T_1\cup \T_2)] = 0$ in $H_2(\T_1\cup \T_2, \partial (\T_1\cup \T_2);\Z)$.

Denote $T=\B^2\times \S^1$ and let $(T;\varphi_1,\varphi_2)$ be the initial package for the Bing double which is the base of the initial package $\fI_{B,n}$, that is, satisfying $\tilde \varphi_i = (\varphi_i\times \id)\circ \psi$ for $i=1,2$. Let also $T_i = \varphi_i(T)$ for $i=1,2$, as before.

Let $g\colon (\B^2,\partial \B^2)\to (T,\partial T)$ be a map having the following properties:
\begin{enumerate}
\item $g(\B^2)\cap T_2 = \emptyset$; 
\item $g$ meets $T_1$ transversely;
\item $g^{-1}(T_1)$ consists of exactly two $2$-cells $D_1$ and $D_2$; 
\item $\langle [g] \rangle = H_2(T,\partial T;\Z)$; 
\item $\langle [g|D_1]\rangle = H_2(T_1,\partial T_1;\Z)$; and 
\item $[g|D_1]=-[g|D_2]$ in $H_2(T_1\cup T_2, \partial (T_1\cup T_2);\Z)$.
\end{enumerate}
Let $y_0\in \S^1$ and define $G\colon (\B^2,\partial \B^2) \to (\T,\partial \T)$ to be the map
\[
x\mapsto (g(x),y_0,\dots,y_0).
\]
Then $G$ satisfies the properties (1)-(6) with $T_1$, $T_2$, and $T$ replaced by $\T_1$, $\T_2$, and $\T$, respectively. In particular,
\[ 
[G|D_1] + [G|D_2] = 0
\]
in $H_2(\T_1 \cup \T_2,\partial (\T_1\cup \T_2);\Z)$.

Let $\hat \Phi \colon D_\omega \to \T$ be an extension of $\Phi$ satisfying $\hat \Phi(D_\omega \setminus \omega) \subset \partial \T$. Since $[\hat \Phi]\ne 0$ in $H_2(\T,\partial \T;\Z)$, there exists $m\ne 0$ for which $[\hat \Phi] = m [G]$ by Lemma \ref{lem:homology}. Thus
\begin{equation}
\label{eq:1}
\Phi -m G + \tau = \partial \sigma
\end{equation}
as $2$-chains, where $\sigma$ is a $3$-chain in $\T$ and $\tau$ is a $2$-chain $\partial \T$. In particular,
\[
\Phi|\Phi^{-1}(\T_1\cup \T_2) - mG|(D_1\cup D_2) + \tau' = 0,
\]
where $\tau'$ is a $2$-chain in $\T\setminus \interior(\T_1\cup \T_2)$. Thus
\[
[\Phi|\Phi^{-1}(\T_1\cup \T_2)] = m[G|(D_1\cup D_2)] = m[G|D_1] + m[G|D_2] = 0
\]
in $H_2(\T_1\cup \T_2, \partial (\T_1\cup \T_2); \Z)$. 
\end{proof}

Finally, before the proof of Proposition \ref{prop:vie}, we note that, for a virtually interior essential map $\Phi \colon (\B^2,\partial B^2)\to (\T,\partial \T)$, the elements in $\Omega(\Phi;\T_1\cup \T_2)$ are not annuli. 

\begin{proof}[Proof of Proposition \ref{prop:vie}]
By Lemma \ref{lem:ie} it suffices to show that $\Omega(\Phi;\T_1\cup \T_2)$ has two innermost components. By Corollary \ref{cor:ie}, $\Omega(\Phi;\T_1\cup \T_2)\ne \emptyset$ and there exists at least one innermost component $\omega_1$. 

Suppose $\omega_1$ is the only innermost component in $\Omega(\Phi;\T_1\cup \T_2)$. We may assume that $\Phi(\omega_1) \subset \T_1$. We show that then there exists a map $\Phi' \colon D_\omega \to \T$ for which $\Phi'|\partial D_\omega = \Phi|\partial D_\omega$ and $\Omega(\Phi';\T_1\cup \T_2)=\{D\}$, where $D=(\Phi')^{-1}(\T_1\cup \T_2)$ is a disk. This is a contradiction. Indeed, on one hand, by Lemma \ref{lem:not_0}, $[\Phi'|D]\ne 0$ either in $H_2(\T_1,\partial \T_1;\Z)$ or in $H_2(\T_2,\partial \T_2;\Z)$. Hence $[\Phi'|D]\ne 0$ in $H_2(\T_1\cup\T_2,\partial (\T_1\cup\T_2);\Z)$. On the other hand, $[\Phi'|D] = [\Phi'|(\Phi')^{-1}(\T_1\cup \T_2)] = 0$ in $H_2(\T_1\cup \T_2,\partial (\T_1\cup \T_2);\Z)$ by Lemma \ref{lem:homology0}.

Let $\bar \Phi \colon (D_\omega,\partial D_\omega) \to (\T,\partial \T)$ be an interior essential map, satisfying $\bar \Phi(D_\omega \setminus \omega) \subset \partial \T$, which is an extension of $\Phi$. Note that, we may assume that $\Phi$, and hence also $\bar \Phi$, meets $\T_1\cup \T_2$ transversally.

Since $\omega_1$ is the unique innermost component, there is an enumeration $\omega_1,\ldots, \omega_k$, satisfying $D_{\omega_j} \subset \interior D_{\omega_{j+1}}$ for each $j=1,\ldots, k-1$, for the components in $\Omega(\Phi_1;\T_1\cup \T_2)$. 

Since $\Phi|\omega_1$ is virtually interior essential by Lemma \ref{lem:ie}, we may assume that $\omega_1$ is a disk, that is, $\omega_1=D_{\omega_1}$. By adapting the argument of Lemma \ref{lem:ie} we may also assume that each $\omega_j$ for $j=2,\ldots, k$ is an annulus. Then $A_j = \mathrm{cl}(D_{\omega_j}\setminus (\omega_j \cup D_{\omega_{j-1}}))$ is an annulus with boundary components $C_j^+ = A_j \cap \omega_j$ and $C_j^-=A_j \cap \omega_{j-1}$ for each $j=2,\ldots, k$. 

We show that, for each $j=2,\ldots, k$, the homomorphism $\pi_1(C^+_j) \to \pi_1(\T_1)$ induced by the inclusion is trivial; note that $C^+_j\subset \partial \T_1$. Thus, for each $j=2,\ldots, k$, there exists interior essential maps $\Phi_j \colon (\B^2,\partial \B^2) \to (\T,\partial \T)$ for which $\Omega(\Phi_j;\T_1\cup \T_2) $ consists of annuli $\omega_{j+1},\ldots, \omega_k$ and the disk $D_{\omega_j}$. In particular, $\Omega(\Phi_k;\T_1\cup \T_2)$ is a disk and we may take $\Phi' = \Phi_k$.

Let $t=\B^2\times \S^1$ and we may assume that there exists $3$-tubes $t_1$ and $t_2$ for which $\T_i = t_i \times (\S^1)^{n-3}$. Thus we may further assume that $\Phi(A_k\cup D_{\omega_{k-1}}) \subset t \times \{x_0\}^{n-3}$, where $x_0\in \S^1$, and we may consider $\Phi|(A_k\cup D_{\omega_{k-1}})$ a map into $t$. Indeed, since $\T_1 \cup \T_2 = (t_1\cup t_2) \times (\S^1)^{n-3}$, it suffices to homotope a lift of $\Phi|(A_k\cup D_{\omega_{k-1}})$ to the cover $(\B^2\times \S^1)\times \R^{n-3}$ to obtain a homotopy of $\Phi|(A_k\cup D_{\omega_{k-1}})$ which ends to a map with the required property.

It suffices to construct $\Phi_2$; the other maps are obtained inductively. Since $A_2\cup D_{\omega_1}$ is a disk, the curve $\Phi|C^+_2$ contracts in $t\setminus \interior (t_2)$. By fixing a homeomorphism, $\S^1\to \partial C^+_2$, we may consider $\Phi|C^+_2$ as a loop. We show first that $\Phi(C^+_2)\subset \partial t_1$. 

Suppose $\Phi(C^+_2)\subset \partial t_2$ and consider a lift $\tilde \Phi$ of $\Phi|C^+_2$ in the universal cover $\pi \colon \B^2\times \R \to t$. We may label, the components $t_{2,k}$ ($k\in \Z$) of $\pi^{-1}t_2$ so that they form a chain with components $t_{1,k}$ of $\pi^{-1}t_1$, that is, $t_{2,k}$ is linked with both $t_{1,k}$ and $t_{1,k+1}$ for each $k\in \Z$; note that with this labelling, homomorphisms $\pi_1(\partial t_{2,k}) \to \pi_1((\B^2\times \R) \setminus \interior( t_{1,j}\cup t_{2,k}))$, for $j=k,k+1$, induced by inclusion, are monomorphisms. Suppose there exists $k_0\in \Z$ for which $\tilde\Phi(C^+_2)\subset t_{2,k_0}$ and let $k_1\in \Z$ be such that $\tilde \Phi(D_{\omega_1}) \subset t_{1,k_1}$; note that $|k_1-k_0|\le 1$. This is a contradiction, since $\tilde \Phi|C^+_2$ is contractible in $\B^2\times \R\setminus \interior(t_{1,k_2}\cup t_{2,k_0})$ for $k_2\in \{k_0,k_0+1\}\setminus \{k_1\}$. Thus $\Phi(C^+_2)\subset \partial t_1$.

By considering $\Phi|C^+_2$ as a loop in $\partial t_1$, we have $[\Phi|C^+_2]=[\alpha]^m[\beta]$ in $\pi_1(\partial t_1)$, where $([\alpha],[\beta])$ is a (standard) basis of $\pi_1(\partial t_1)$, that is, $\alpha$ contracts in $t_1$ and $\beta$ in $t\setminus \interior (t_1)$. Using again the fact that $\Phi|C^+_2$ contracts in $t\setminus t_2$, we conclude that $\ell =0$, that is, $[\Phi|C^+_2] = [\alpha]^m$. In particular, there exists a map $D_{\omega_2}\to t_1$ which extends $\Phi|\omega_2$. This concludes the construction of $\Phi_2$ and the proof.
\end{proof}

\section{Modulus estimates}
\label{sec:modulus}

In this section we show a lower bound for moduli of certain families of $(n-2)$-tori in the Semmes space $(\S^n,d)$ and an upper bound for the corresponding families in the Euclidean sphere $\S^n$. For the statement, we introduce some terminology.

Let $T$ be an $n$-tube in $\S^n$. An $(n-2)$-torus $t\subset T$ in $T$ is a \emph{core torus of $T$} if there exists a homeomorphism $\theta \colon \B^2 \times (\S^1)^{n-2}\to T$ for which $t = \theta(\{0\}\times (\S^1)^{n-2})$. 

Let $\fI_{B,n}$ be an initial package for the Bing-Blankinship necklace and $(\T_w)_{w\in \cW_2}$ the associated defining tree. For each $w\in \cW_2$, we denote by $\cS_w$ the family of all core tori $t$ in $\T_w$ for which $t\subset \T_w\setminus (\T_{w1}\cup \T_{w2})$.

\subsection{Modulus lower bound in the Semmes space}
 
Let $\vartheta \colon \T_\emptyset\to \S^n$ be a smooth embedding and let $\hat \rho \colon \S^n\to \S^n$ be the Bing--Blankinship shrinking map in Proposition \ref{prop:BB_Bing}. Let $(T_w)_{w\in \cW_2}$ be the ordered tree with $T_w = \hat\rho(\vartheta(\T_w))$ for $w\in \cW_2$ and $\sS_w = \{\hat \rho(t)  \colon t\in \cS_w\}$ for each $w\in \cW_2$. Let $d$ be a Semmes metric as in Corollary \ref{cor:metric_E} with the scaling constant $\lambda \in (0,1)$. 

We summarize in the following lemma the basic properties of the metric space $(\S^n,d)$ and the families $(\sS_w)_{w\in \cW_2}$ which will be used in the forthcoming discussion. These properties are direct consequences of the shrinking map $\hat \rho$ and the construction of the metric $d$; see Section \ref{sec:BB} and the references therein. 

For brevity, we call $(\S^n,\fI_{B,n},\rho,\lambda,d)$ the \emph{data of the Semmes space $(\S^n,d)$}, and $(\S^n,\fI_{B,n},\rho,\lambda,d; (T_w)_w,(\sS_w)_w)$ an \emph{extended data}; note that $(T_w)_w$ and $(\sS_w)_w$ are fully determined by the other data.

\begin{lem}
\label{lemma:basic_properties}
Let $(\S^n,\fI_{B,n},\rho,\lambda,d;(T_w)_w,(\sS_w)_w)$ be an extended data. Then
\begin{itemize}
\item[(i)] for each $w\in \cW_2$, the family $\sS_w$ consists of $(n-2)$-tori,
\item[(ii)] $\lim_{k\to \infty} \sup_{|w|\ge k}\diam_d{T_w} = 0$, and
\item[(iii)] there exists $L\ge 1$ and $\delta>0$ depending only on the data so that, for each $w\in \cW_2$, $B_d(\partial T_w,\delta\lambda^{|w|+1})\cap T_w$ is (smoothly) $(L,\lambda^{|w|})$-quasisimilar to $B_d(\partial T_\emptyset,\delta)$.
\end{itemize}
\end{lem}

Note that, here and in what follows, we indicate the use of the Semmes metric $d$ by a subscript in the metric notions such as diameter and neighborhood.

The modulus lower bound for families $\sS_w$ in the Semmes space $(\S^n,d)$ is a direct corollary of (iii) in Lemma \ref{lemma:basic_properties}.

\begin{prop}
\label{prop:Semmes_modulus}
Let $(\S^n,\fI_{B,n},\rho,\lambda,d; (T_w)_w, (\sS_w)_w)$ be an extended data. Then there exists $c_0>0$ depending only the data so that 
\[
\Mod_{\frac{n}{n-2},d}(\sS_w) \ge c_0
\] 
for each $w\in \cW_2$.
\end{prop}
\begin{proof}
Let $\delta>0$ be as in (iii) in Lemma \ref{lemma:basic_properties} and let $\sS_0\subset \sS_\emptyset$ be the subfamily of core tori $t$ contained in the $\delta$-neighborhood $\Omega = B_d(\partial T_\emptyset,\delta)\cap T_\emptyset$ of $\partial T_\emptyset$ in $T_\emptyset$. By by the uniform quasisimilarity, it suffices to show that $\Mod_{\frac{n}{n-2},d}(\sS_0) \ge c_0$, where $c_0>0$ depends only on the data. 

By properties of the metric $d$, there exists $L_0\ge 1$ depending only on the data and an $L_0$-bilipschitz diffeomorphism $f \colon A \times (\mathbb{S}^1)^{n-2} \to \Omega$, where $A=\B^2\setminus \B^2(1-\delta)$. Let $f^{-1}\sS_0 = \{ f^{-1}t \colon t\in \sS_0\}$. Since $f$ is $L_0^{2n}$-quasiconformal, the quasi-invariance of the conformal modulus yields the estimate
\[
\Mod_{\frac{n}{n-2},d}(\sS_0) \ge c_0 \Mod_{\frac{n}{n-2}}(f^{-1}\sS_0), 
\]
where $c_0>0$ depends only on $L_0$ and $n$. Thus it suffices to show that the family $\sS_{A}=\{ \{x\}\times (\S^1)^{n-2} \colon x\in A\}\subset f^{-1}\sS'$ has modulus lower bound.

Let $\rho\colon A\times (\S^1)^{n-2}\to \R$ be an admissible function for $\sS_{A}$. By H\"older's inequality, 
\begin{eqnarray*}
&& \int_{A\times (\S^1)^{n-2}} \rho(x)^{\frac{n}{n-2}} \mathrm{d}\haus^n(x) \\
&&\qquad =
\int_A \left( \int_{(\S^1)^{n-2}} \rho(z,y)^{\frac{n}{n-2}} \mathrm{d}\haus^{n-2}(y)\right) \mathrm{d}\haus^2(z) \\
&&\qquad \ge \haus^{n-2}\left((\S^1)^{n-2}\right)^{\frac{-2}{n-2}} \int_A\left(\int_{(\S^1)^{n-2}} \rho(z,y) \mathrm{d}\haus^{n-2}(y)\right)^{\frac{n}{{n-2}}}\mathrm{d}\haus^2(z)\\
&&\qquad \ge \haus^{n-2}\left((\S^1)^{n-2}\right)^{\frac{-2}{n-2}} \haus^2(A).
\end{eqnarray*}
This concludes the proof 
\end{proof}

\subsection{Modulus upper bound in the Euclidean sphere}

For the modulus upper bound, we pass to a non-smooth setting in the following sense. Let $(\T_w)_{w\in \cW_2}$ be the Bing--Blankinship defining tree associated to the initial package $\fI_{B,n}$ and let $\vartheta'\colon \T_\emptyset \to \S^n$ be an embedding. We may assume that $\vartheta'\T_\emptyset \subset \R^n \subset \S^n$.

Let also $\varrho' \colon \S^n \to \S^n$ be a Bing--Blankinship map as in Proposition \ref{prop:BB_Bing} with the exception that $\varrho'|\S^n\setminus \cS(\fI_{B,n})$ is merely a homeomorphism; note that we do not assume $\vartheta'$ to be smooth. We denote now $T'_w = (\varrho'\circ \vartheta')\T_w$ and $\sS'_w = (\varrho'\circ \vartheta')\cS_w$ for each $w\in \cW_2$. As summarized in Lemma \ref{lemma:basic_properties}, we again have that $\diam{T'_w}\to 0$ as $|w|\to \infty$ and families $\sS_w$ consist of $(n-2)$-tori.

The modulus upper bound now reads as follows. 
\begin{prop}
\label{prop:HW}
Let $\alpha \colon \S^1\to \R^n$, $z\mapsto (z,0,\ldots,0)$. Suppose $|\alpha|$ is in the complement of $T'_\emptyset$ and suppose $\alpha$ is not homotopic to a constant map in $\R^n\setminus T'_\emptyset$. Let $\delta = \mathrm{dist}(|\alpha|,T'_\emptyset)$. Then there exists $C_1>0$ depending only on $n$ so that, for each $k\in \N$, there exists $w_k\in \cW_2$ of length $k$ for which
\[
\Mod_{\frac{n}{n-2}}(\sS'_{w_k}) \le C_1 \left(\frac{\diam{T'_{w_k}}}{\delta}\right)^n. 
\]
\end{prop}

Although the proof is merely a part of the proofs of \cite[Proposition 4.5]{HW} and \cite[Theorem 10.1]{PW}, we recall the argument.

\begin{proof}[Proof of Proposition \ref{prop:HW}]
We show first that, for each $k\in \N$, there exists $w_k\in \cW_2$ of length $k$ for which
\begin{equation}
\label{eq:inf_area}
\inf_{t\in \sS_{w_k}} \haus^{n-2}(t) \ge c_{n-2} \delta^{n-2},
\end{equation}
where $c_{n-2} = \haus^{n-2}(\B^{n-2})$.

Having \eqref{eq:inf_area} at our disposal, the claim follows by the standard modulus estimate. Indeed, the function $\rho = (c_{n-2} \delta^{n-2})^{-1}\chi_{T'_{w_k}}$ is an admissible function for $\sS_{w_k}$ and
\begin{eqnarray*}
\Mod_{\frac{n}{n-2}}(\sS_{w_k}) &\le& \int_{\R^n}\rho^{\frac{n}{n-2}} \, d\haus^{n} = \frac{(\diam{T'_{w_k}})^n}{(c_n\delta^{n-2})^{\frac{n}{n-2}}} = C_1 \left( \frac{\diam{T'_{w_k}}}{\delta}\right)^n.
\end{eqnarray*}

To prove the estimate \eqref{eq:inf_area}, let $k\in \N$ and $\varepsilon>0$. For each $w\in \cW_2$ of length $|w|=k$, we fix $t_w \in \sS'_w$ for which
\[
\haus^{n-2}(t_w) \ge \inf_{t\in \sS'_w} \haus^{n-2}(t) - \varepsilon.
\]

We claim first that 
\begin{equation}
\label{eq:intersections}
\#((\bigcup_{|w|=k}t_w)\cap (\B^2\times \{j\})) \ge 2^k
\end{equation}
for each $j\in \B^{n-2}(\delta)$. Indeed, let $j\in \B^{n-2}(\delta)$ and consider the map $\zeta_j \colon \B^2\to \S^n$, $u \mapsto (u,j)$. Since $\zeta_j|\S^1$ is homotopic to $\alpha$ in $\S^n\setminus T'_\emptyset$, $\zeta_j|\S^1$ is not null-homotopic in $\R^n\setminus T'_\emptyset$. Thus there exists a domain $D_j\subset \B^2$ so that $\Psi_j = \zeta_j|D_j \colon (D_j,\partial D_j) \to (T'_\emptyset,\partial T'_\emptyset)$ is virtually interior essential. 

The count of the intersections $\Psi_j(D_j) \cap \bigcup_{|w|=k} t_w$ reduces to Proposition \ref{prop:vie} as follows. Let $q\colon \S^n\to \S^n/\mathrm{BB}$ be the quotient map and $h'\colon \S^n \to \S^n/\mathrm{BB}$ the homeomorphism satisfying 
\[
\xymatrix{
\S^n \ar[rr]^{\varrho'} \ar[dr]_q & & \S^n \ar[dl]^{h'}_{\approx} \\
& \S^n/\mathrm{BB} & }
\]
It is now easy to find a virtually interior essential map $\Phi_{j,k}\colon (D_j,\partial D_j) \to \vartheta'(\T_\emptyset), \partial \vartheta'(\T_\emptyset))$ for which $(\varrho'\circ \Phi_{j,k})|D_{j,k} = \Psi_j|D_{j,k}$, where $D_{j,k} = \Psi_j^{-1}(\R^n\setminus \bigcup_{|w|=k+1} T'_w)$; we refer to \cite[Lemma 10.2]{PW} for a detailed argument.  By Proposition \ref{prop:vie}, 
\[
\#\Omega(\Phi_{j,k}; \bigcup_{|w|=k} \vartheta'(\T_w)) \ge 2^k.
\]
Since each element in $\Omega(\Phi_{j,k}; \bigcup_{|w|=k} \vartheta'(\T_w))$ meets $(\varrho')^{-1}\bigcup_{|w|=k} t_w$, inequality (\ref{eq:intersections}) follows.

By the co-area formula \cite[Theorem 2.10.25]{Federer}, we have
\begin{eqnarray*}
\sum_{|w|=k} \haus^{n-2}(t_w)) &=& \haus^{n-2}(\bigcup_{|w|=k} t_w) \\
&\ge& \haus^{n-2}((\bigcup_{|w|=k} t_w)\cap (\B^2 \times \B^{n-2}(\d)))\\
&\ge& \int_{\B^{n-2}(\delta)} \# ((\bigcup_{|w|=k} t_w)\cap (\B^2\times\{j\})) \, d\haus^{n-2}(j)\\
&\ge& 2^k \haus^{n-2}(\B^{n-2}(\delta)) = 2^k c_{n-2} \delta^{n-2},
\end{eqnarray*}
where $c_{n-2} = \haus^{n-2}(\B^{n-2})$. Thus \eqref{eq:inf_area} holds.
\end{proof}

\section{Analog of Semmes's theorem in higher dimensions}\label{sec:proof_main}

Before discussing the proof of the main theorem (Theorem \ref{thm:point}), we give a short proof of the following result.

\begin{thm}
\label{thm:main}
For each $n\ge 3$ and $\lambda \in (0,2^{-1/n})$ there exists a Cantor set $E\subset \S^n$ and an almost smooth Ahlfors $n$-regular and LLC Semmes metric $d$ with scaling constant $\lambda$ and singular set $E$ for which there is no quasiconformal homeomorphism $(\S^n,d) \to \S^n$.
\end{thm}

This result parallels Semmes's non-parametrization theorem for metrics on $\S^3$ in \cite{SemmesS:Goomsw} in the sense that the Ahlfors $n$-regularity of the space $(\S^n,d)$ is the only condition which restricts the scaling constant $\lambda$ of the metric $d$. In results of Heinonen and Wu \cite{HW} the method of stabilization poses additional restriction for $\lambda$. We refer to \cite[Section 13]{PW} for a discussion and a general necessary condition relating the scaling parameter to the complexity of the defining sequence in quasiconformal non-parametrization questions in the stabilized case.

The remaining ingredient in the proof of Theorem \ref{thm:main} that we have not discussed yet is a uniform straightening lemma from \cite{PW} for quasisymmetrically embedded collared circles. Recall that $\R^n$ embeds in $\S^n$ via the stereographic projection. 

\begin{lem}[{\cite[Proposition 11.1]{PW}}]
\label{lemma:straigthening}
Let $n\geq 4$ and $h\colon \S^1 \times \B^{n-1} \to \R^{n}$ an $\eta$-quasisymmetric embedding. Then there exists a constant $\delta_0>0$ and a homeomorphism $\tilde \eta\colon [0,\infty) \to [0,\infty)$ both depending only on $n$ and $\eta$, and an $\tilde\eta$-quasisymmetric homeomorphism $\chi \colon \S^n \to \S^n$ for which 
\begin{itemize}
\item[(a)] $\S^1 \times \B^{n-2}(\delta_0) \subset (\chi\circ h)(\S^1\times \B^{n-1})$,
\item[(b)] the maps $\S^1\to \R^n$ defined by $z\mapsto (z,0)$ and $z\mapsto (\chi\circ h)(z,0)$ are homotopic in $(\chi\circ h)(\S^1\times \B^{n-1})$.
\end{itemize}
\end{lem}

\begin{proof}[Proof of Theorem \ref{thm:main}]
Let $d$ be a Semmes metric on $\S^n$ having an extended data $(\S^n,\fI_{B,n},\vartheta, \varrho,\lambda,d; (T_w)_w,(\sS_w)_w)$, where $\lambda \in (0,2^{-1/n})$, and a singular set $E$ which is the Bing--Blankinship Cantor set associated to this data. By Lemmas \ref{lem:Ahlfors} and \ref{lem:LLC}, $(\S^n,d)$ is Ahlfors $n$-regular and linearly locally contractible. Thus it remains to show that the metric sphere $(\S^n,d)$ is not quasiconformal to the Euclidean sphere $\S^n$.

Suppose there exists a quasiconformal map $f\colon (\S^n,d)\to \S^n$. Since $(\S^n,d)$ is a Loewner space, $f$ is a quasisymmetry. By Proposition \ref{prop:Semmes_modulus} and the quasi-invariance of the conformal modulus, 
\[
\inf_{w\in \cW_2} \Mod_{\frac{n}{n-2}}(f\sS_w) > 0.
\]
By Lemma \ref{lemma:straigthening}, we may also assume that 
\begin{itemize}
\item[(i)] $\S^1\times \{0\} \subset \S^n \setminus \overline{fT_\emptyset}$ and
\item[(ii)] the map $\alpha \colon \S^1\to \S^n$, $z\mapsto (z,0)$, is not contractible in $\S^n\setminus fT_\emptyset$. 
\end{itemize}
Indeed, let $\gamma \colon \S^1 \to \S^n$ be a smooth simple curve so that $\dist(|\gamma|, T_\emptyset)>0$ and $\gamma$ is not contractible in $\S^n\setminus T_\emptyset$. By Lemma \ref{lemma:TNT} there exists a quasisymmetric embedding $h_0 \colon \S^1\times \B^{n-1} \to (\S^n,d)$ for which $h_0(z,0) = \gamma(z)$. Then $h = f\circ h_0 \colon \S^1\times \B^{n-1} \to \S^n$ is a quasisymmetric embedding. Let now $\chi \colon \S^n \to \S^n$ be a quasisymmetric homeomorphism as in Proposition \ref{lemma:straigthening}. Then conditions (i) and (ii) are satisfied by the map $\chi\circ f$.

Thus, by Proposition \ref{prop:HW}, there exists a sequence $(w_k)_{k\in \N}$ in $\cW_2$ for which 
\[
\Mod_{\frac{n}{n-2}}(f\sS_{w_k}) \to 0
\]
as $k\to \infty$. This contradiction concludes the proof. 
\end{proof}

\section{Proof of Theorem \ref{thm:point}}\label{sec:proof}

The proof of Theorem \ref{thm:point} is a slight modification of the proof of Theorem \ref{thm:main} along the lines of construction of the Riemannian manifold $\tilde M$ in \cite[p.206]{SemmesS:Goomsw} and uses the uniform bounds in Lemma \ref{lemma:straigthening}. Thus we merely indicate the steps.

\subsection{The metric}

We fix first a sequence $(B_k)$ of pair-wise disjoint Euclidean balls in $\R^n$ which converge to the origin; for example, we may take $B_k=\B^n(x_k,r_k)$, where $x_k = e_1/k$ and $r_k \le |x_k|/10$. Let also $\vartheta \colon \T \to \S^n$ be a smooth embedding. 

For each $k\in\N$, we fix an $n$-tube $T^{(k)}=\alpha_k(T)\subset B_k$, where $T=\vartheta(\T)$ and $\alpha_k \colon \R^n\to \R^n$ is a similarity. We plant a finite decomposition tree $\cT_k$ into each $T^{(k)}$ as follows.

For each $k\in \N$, let $\cW^{(k)}_2$ be the collection of all words of length at most $k$. We define finite decomposition trees $\cT^{(k)}$ by $\cT^{(k)} = (\alpha_k(\vartheta(\T_w)))_{w\in \cW^{(k)}_2}$. Note that $\cT^{(k)}$ is in fact a subtree of a decomposition tree of the initial package $(T^{(k)}; \tilde \varphi^{(k)}_1,\tilde \varphi^{(k)}_2)$, where the embedding $\tilde \varphi^{(k)}_i$ is obtained by conjugating $\tilde \varphi_i$ with mappings $\vartheta$ and $\alpha_k$ in the obvious manner. We denote $T^{(k)}_w = \tilde \varphi^{(k)}_w(T^{(k)})$ for $w\in \cW^{(k)}_2$.

Let now $\lambda \in (0,2^{-1/n})$. By applying the construction of the Semmes embedding (Proposition \ref{prop:BB_Semmes}) independently in each ball $B_k$ with respect to the finite decomposition tree $\cT^{(k)}$, we obtain a map $\tilde \rho \colon \S^n \to \S^{n+1}$ which is a smooth embedding in $\S^n\setminus \{0\}$ and for which 
\begin{itemize}
\item[(i)] $\tilde \rho(x)=x$ for $x\not \in \bigcup_{k\in \N}B_k$;
\item[(ii)] there exists $L\ge 1$ so that, for each $k\ge \N$ and $w\in \cW^{(k)}_2$ of length $|w|<k$, 
\[
\frac{\lambda^{|w|}}{L}|x-y| \le |\tilde \rho\circ \tilde\varphi_{w}^{(k)}(x)-\tilde \rho\circ \tilde\varphi_{w}^{(k)}(y)| \le L\lambda^{|w|}|x-y|
\]
for each $x,y\in \tilde\varphi_{w}^{(k)}(T^{(k)}\setminus T^{(k)}_1\cup T^{(k)}_2)$; and
\item[(iii)] for each $w\in \cW_2$ of length $k$, $\tilde\rho(T^{(k)})$ is similar to $T^{(k)}$.
\end{itemize}
Indeed, the embedding $\tilde\rho$ is an intermediate stage in the construction of the embedding in Proposition \ref{prop:BB_Semmes}. We refer to \cite[Lemma 3.21]{SemmesS:Goomsw} and \cite[Section 6]{PW} for details. Let $d$ be the completion of the length metric $d_g$ associated to the Riemannian metric $g = \tilde \rho^*g_0$, where $g_0$ is the Riemannian metric on $\S^{n+1}$.

The space $(\S^n,d)$ is Ahlfors $n$-regular and LLC. Indeed, we note first that the method of the proof of \cite[Proposition 7.8]{PW} (here Lemma 2.1) readily applies also to the embedding $\tilde \rho|\S^n\setminus \{0\}$ and we conclude that
\begin{equation}
\label{eq:Ahl}
\haus^{n}_d(B(x,r)) \approx r^n
\end{equation}
for each ball $B(x,r)\subset \S^n$ not containing the origin. Since also 
\[
\haus^n_d(T^{(k)})\approx \haus^n(T^{(k)})
\]
uniformly in $k$, the argument of \cite[Proposition 7.8]{PW} shows that \eqref{eq:Ahl} holds for all balls $B(x,r)$ in $(\S^n,d)$. Thus $(\S^n,d)$ is Ahlfors $n$-regular. The linear local contractibility of $(\S^n,d)$ is argued along the lines of the proof of \cite[Proposition 7.9]{PW}. 

\subsection{Non-parametrizability}

It remains to prove the non-existence of a quasiconformal homeomorphism $(\S^n,d)\to \S^n$. Suppose towards contradiction that such homeomorphism $f\colon (\S^n,d)\to \S^n$ exists.

For each $k\in \N$ and $w\in \cW^{(k)}_2$ of length $k-1$, let $\sS^{(k)}_w$ be the family $\alpha_k(\vartheta \cS_w)$ of $(n-2)$-tori. By construction of the metric $d$, 
\[
\Mod_{\frac{n}{n-2},d}(\sS^{(k)}_w) \approx \Mod_{\frac{n}{n-2},d_\infty}(\sS_w),
\]
for each $w\in \cW^{(k)}_2$ and $k\in \N$, where $(\S^n,d_\infty)$ is the Semmes space in Theorem \ref{thm:main} and $\sS_w$ the family of surfaces in the proof of Theorem \ref{thm:main}. Thus
\[
\inf_{\substack{k\in \N\\ w\in \cW^{(k)}_2}} \Mod_{\frac{n}{n-2},d}(\sS^{(k)}_w) > 0.
\]

On the other hand, by the properties of the metric $d$, there exists a neighborhood $\Omega$ of $\partial T$ for which each $\alpha_k|\Omega$ is a similarity in metric $d$. Thus, for each $k\in\N$, we can fix smooth simple curves $\gamma_k = \a_k\circ \gamma \colon \S^1 \to \a_k(\Omega)$, where $\gamma$ is a smooth simple curve in $\Omega$ so that $\dist(|\gamma|,\partial T) >0$ and $\gamma$ is not contractible in $\S^n\setminus T$. Note that each $\gamma_k$ is not contractible in $\S^n \setminus T^{(k)}$ and $\dist(|\gamma_k|,T^{(k)}) = C_k\dist(|\gamma|,\partial T)$, where $C_k$ is the similarity constant of $\a_k$. 

We may now fix a homeomorphism $\eta'\colon [0,\infty) \to [0,\infty)$ and, for each $k\in \N$, an $\eta'$-quasisymmetric embedding $h_k \colon \S^1\times \B^{n-1}\to (\S^n,d)$ for which
\begin{itemize}
\item[(1)] $h_k(\S^1\times \B^{n-1}) \subset \S^n\setminus \bigcup_{\ell} T^{(\ell)}$ and
\item[(2)] the map $\S^1 \to \S^n$, $z\mapsto h_k(z,0)$, is not contractible in $\S^n\setminus T^{(k)}$.
\end{itemize}

Thus, by Lemma \ref{lemma:straigthening}, there exists a homeomorphism $\eta''\colon [0,\infty)\to [0,\infty)$, a constant $\delta>0$ and, for each $k\in \N$, an $\eta$-quasisymmetric map $\chi_k \colon \S^n\to \S^n$ for which
\begin{itemize}
\item[(i)] $\S^1\times \B^{n-2}(\delta) \subset \S^n\setminus(\chi_k \circ f \circ h_k)(T^{(k)}_\emptyset)$, and
\item[(ii)] the map $\S^1\to \S^n$, $z\mapsto (\chi_k \circ f\circ h_k)(z)$, is homotopic to $z\mapsto (z,0)$ in $(\chi_k \circ f \circ h_k)(\S^1\times \B^{n-1})$.
\end{itemize}

Thus, by Proposition \ref{prop:HW}, there exists a sequence $(w_k)$ in $\cW_2$ for which 
\[
\Mod_{\frac{n}{n-2}}(f\sS^{(k)}_{w_k}) \to 0
\]
as $k\to \infty$. This is a contradiction. Thus there is no quasiconformal homeomorphism $(\S^n,d)\to \S^n$. This completes the proof of Theorem \ref{thm:point}.

\bibliographystyle{abbrv}
\bibliography{Whitehead}

\end{document}